\documentclass{article}
\usepackage{amsmath}
\usepackage{amsfonts}
\usepackage{graphicx}
\usepackage[pdftex]{hyperref}
\usepackage[dutch,english]{babel}
\usepackage{amsmath,amssymb}
\usepackage{amsthm}
\usepackage[usenames,dvipsnames]{color}
\usepackage{tikz}
\usepackage{accents}
\usepackage{xcolor}
\usepackage{hyperref}
\usepackage{mathtools}
\usepackage{multicol}
\usepackage{comment}

\usepackage{subcaption}

\usepackage{tikz}
\usetikzlibrary{shadows, arrows}

\DeclareMathOperator{\Arg}{Arg}

\theoremstyle{plain}
\newtheorem{thr}{Theorem}[section]

\newtheorem{lem}[thr]{Lemma}
\newtheorem{prop}[thr]{Proposition}
\newtheorem{cor}[thr]{Corollary}
\theoremstyle{definition}
\newtheorem{defi}[thr]{Definition}
\newtheorem{ex}[thr]{Example}
\theoremstyle{remark}
\newtheorem{remk}{Remark}
\theoremstyle{remark}

\newcommand{\field}[1]{\mathbb{#1}}
\newcommand{\R}{\field{R}}

\newcommand{\C}{\field{C}}

\definecolor{wred}{rgb}{0.7,0.18,0.12}
\definecolor{wgreen}{rgb}{0.1,0.53,0.37}

\numberwithin{equation}{section}

\def\barray{\begin{eqnarray*}}             \def\earray{\end{eqnarray*}}
\def\beq{\begin{equation}} \def\eeq{\end{equation}}

\title{A parametrisation method for high-order phase reduction in coupled oscillator networks}

\author{S\"oren von der Gracht\thanks{Department of Mathematics, Paderborn University, Germany, \href{mailto:soeren.von.der.gracht@uni-paderborn.de}{soeren.von.der.gracht@uni-paderborn.de}}, 
Eddie Nijholt\thanks{\mbox{Department of Mathematics, Imperial College London, United Kingdom, \break \href{mailto:eddie.nijholt@gmail.com}{eddie.nijholt@gmail.com}}}, Bob Rink\thanks{\mbox{Department of Mathematics, Vrije Universiteit Amsterdam, The Netherlands, \href{mailto:b.w.rink@vu.nl}{b.w.rink@vu.nl}}}}

\date{\today}

\excludecomment{percent}

\begin{document}

\maketitle

\begin{abstract} 
We present a novel method for high-order phase reduction in networks of weakly coupled oscillators and, more generally, perturbations of reducible normally hyperbolic (quasi-)periodic tori. Our method works by computing an asymptotic expansion for an embedding of the perturbed invariant torus, as well as for the reduced phase dynamics in local coordinates. Both can be determined to arbitrary degrees of accuracy, and we show that the phase dynamics may directly be obtained in normal form. 
We apply the method to predict remote synchronisation in a chain of coupled Stuart-Landau oscillators. 
\end{abstract}

\section{Introduction}

Many systems in  science and engineering consist of coupled periodic processes. Examples vary from the motion of the planets, to the synchronous flashing of fireflies \cite{fireflies}, and from the activity of neurons in the brain \cite{golomb}, to power grids and electronic circuits. The functioning and malfunctioning of these coupled systems is often determined by a form of collective behaviour of its constituents, perhaps most notably their synchronisation \cite{kurths, pikovsky}. For example, synchronisation of neurons  plays a critical role in cognitive
processes \cite{Nicolette, palva}. 

In this paper, we consider the situation where the coupling between the periodic processes is weak, a case that is amenable to rigorous mathematical analysis. Specifically, we assume that the evolution of the processes can be modelled by a system of differential equations of the form
\begin{align} \label{coupledoscillators}
\dot x_j = F_j(x_j) + \varepsilon \, G_j(x_1, \ldots, x_m)   \ \mbox{for}\ x_j \in \mathbb{R}^{M_j} \ \mbox{and}\ j=1, \ldots, m\, .
\end{align}
The vector fields $F_j: \mathbb{R}^{M_j} \to \mathbb{R}^{M_j}$ in \eqref{coupledoscillators} determine the dynamics of the uncoupled oscillators:  we assume that each $F_j$ possesses a hyperbolic $T_j$-periodic orbit $X_j(t)$. In the uncoupled limit---when $\varepsilon=0$---equations \eqref{coupledoscillators} thus admit a normally hyperbolic periodic or quasi-periodic invariant torus $\mathbb{T}_0 \subset \mathbb{R}^M$ (where $M := M_1 + \ldots + M_m$), consisting of the product of these periodic orbits. 
The functions $G_j$ in \eqref{coupledoscillators} model the interaction between the oscillators, for example through a (hyper-)network. The interaction strength $0 \leq \varepsilon \ll 1$ is assumed small, so that the unperturbed torus $\mathbb{T}_0$ persists as an invariant manifold $\mathbb{T}_{\varepsilon}$ for \eqref{coupledoscillators}, depending smoothly on $\varepsilon$, as is guaranteed by F\'enichel's theorem \cite{fenichel1979, wechselberger2020}. 

The process of finding the  equations of motion that govern the dynamics on the persisting  torus $\mathbb{T}_{\varepsilon}$ is usually referred to as {\it phase reduction} \cite{wilson, nakao, Pietras}. Phase reduction has proved a powerful tool in the study of the synchronisation of coupled oscillators, especially because it often realises a considerable reduction of the dimension---and hence  complexity---of the system. 
Various methods of phase reduction have been introduced over the past decades, the most well-known appearing perhaps in the work on chemical oscillations of Kuramoto \cite{Kuramotobook}. We refer to \cite{Pietras} for an extensive overview of established phase reduction techniques, and refrain from providing an overview of these methods here. 

Most existing phase reduction methods provide a first-order approximation of the  dynamics on the persisting invariant torus in terms of the small coupling parameter. However, there are various instances where such a first-order approximation is insufficient, see \cite{BickBoehle,  rosenblum, Pazo, numericalphasereduction, nijholt2022emergent}, in particular when the first-order reduced dynamics is structurally unstable. For instance, it was observed in \cite{rosenblum} that ``remote synchronisation'' \cite{remotestar} cannot  be analysed with first-order methods. More accurate ``high-order phase reduction'' techniques (that go beyond the first-order approximation) have only been introduced very recently \cite{BickBoehle, Gengel, Pazo}.  They have already been applied successfully, for example to predict remote synchronisation \cite{rosenblum}. However, to the best of our knowledge, mathematically rigorous high-order phase reduction methods have only been derived in the special case that the unperturbed oscillators are either Stuart-Landau oscillators  \cite{Gengel, Pazo} or deformations thereof \cite{AshwinRodrigues, BickBoehle}. In that setting, phase reduction can be performed by computing an expansion of the phase-amplitude relation that defines the invariant torus. However, this procedure does not generalise to arbitrary systems of the form \eqref{coupledoscillators}.   

This paper  presents a novel method for high-order phase reduction, that applies to general coupled oscillator systems of the form \eqref{coupledoscillators}. 
Our method works by computing an expansion (in the small parameter $\varepsilon$) of an embedding
$$e: (\mathbb{R}/2\pi\mathbb{Z})^m \to \mathbb{R}^M\, $$
 of the persisting invariant torus $\mathbb{T_{\varepsilon}}$. In addition, it computes 
 an expansion of the dynamics on $\mathbb{T}_{\varepsilon}$ in local coordinates, in the form of a so-called ``reduced phase vector field''  $${\bf f}: (\mathbb{R}/2\pi\mathbb{Z})^m \to \mathbb{R}^m$$ 
 on the standard torus $(\mathbb{R}/2\pi\mathbb{Z})^m$.  We find these $e$ and ${\bf f}$ by solving a so-called ``conjugacy equation''.  
 Our method is thus inspired by the work of De la Llave et al. \cite{parameterisation_book}, who popularised the idea of finding invariant manifolds by solving conjugacy equations. In fact, this idea was  used in \cite{HaroCanadell2} to design a quadratically convergent iterative scheme for finding normally hyperbolic invariant tori. However, in \cite{HaroCanadell2} these tori are required to carry Diophantine quasi-periodic motion, not only before but also after the perturbation.    

 The phase reduction method presented in this paper is more similar in nature to the parametrisation method developed in \cite{BobIanMartin}. There the idea of parametrisation is used to calculate expansions of slow manifolds  and their flows in geometric singular perturbation problems  \cite{wechselberger2020}. Just like the method in \cite{BobIanMartin}, the phase reduction method presented here yields  asymptotic expansions to finite order, but it poses no restrictions on the nature of the dynamics on the invariant torus.

We now sketch the idea behind our method. Let us write
 ${\bf F}_0$ for the vector field on $\mathbb{R}^M = \mathbb{R}^{M_1}\times \ldots \times \mathbb{R}^{M_m}$ that governs the dynamics of the uncoupled oscillators in \eqref{coupledoscillators}, that is,  
\begin{align}\label{F0def}
{\bf F}_0(x_1, \ldots, x_m) := (F_1(x_1), \ldots, F_m(x_m))\, .
\end{align}
Our starting point is an embedding of the invariant torus $\mathbb{T}_0$ for this ${\bf F}_0$. Recall our assumption that every $F_j$  possesses a hyperbolic periodic orbit $X_j(t)$ of minimal period $T_j>0$. We denote the frequency of this orbit by $\omega_j:= \frac{2\pi}{T_j}$.  
An obvious embedding of $\mathbb{T}_0$ is the map 
$e_{0}:   (\mathbb{R}/2\pi\mathbb{Z})^m      \to \mathbb{R}^{M}$
  defined by 
\begin{align}\label{unperturbedembedding}
e_0(\phi) = e_0(\phi_1, \ldots, \phi_m) := \left(X_1\left( \omega_1^{-1}\phi_1 \right), \ldots, X_m\left(\omega_m^{-1} \phi_m \right) \right) \, .
\end{align}
In fact, this $e_0$ sends the periodic or quasi-periodic solutions of the ODEs
$$\dot \phi = \omega := (\omega_1, \ldots, \omega_m) $$ 
on 
$(\mathbb{R}/2\pi\mathbb{Z})^m$ 
 to integral curves of ${\bf F}_0$. In other words---see also Lemma \ref{basiclemma} below---it  satisfies the conjugacy equation
 $$e_0' \cdot \omega = {\bf F}_0 \circ e_0  \, .$$
 The idea is now that we search for an asymptotic approximation of an embedding of the persisting torus $\mathbb{T}_{\varepsilon}$ by solving a similar conjugacy equation. 
We do this by making a series expansion ansatz for such an embedding,  of the form 
$$e = e_0 + \varepsilon e_1 +  \varepsilon^2 e_2 + \ldots : (\mathbb{R}/2\pi\mathbb{Z})^m \to \mathbb{R}^M\, , $$  as well as for  
 a reduced phase vector field $${\bf f} = \omega + \varepsilon {\bf f}_1 +  \varepsilon^2 {\bf f}_2  + \ldots: (\mathbb{R}/2\pi\mathbb{Z})^m \to \mathbb{R}^m\, .$$ 
Indeed, writing ${\bf F} = {\bf F_0} + \varepsilon {\bf F}_1: \mathbb{R}^M\to\mathbb{R}^M$, with 
${\bf F}_0$ as above, and $${\bf F}_1(x) := (G_1(x), \ldots, G_m(x))\,  $$
 denoting the coupled part of  \eqref{coupledoscillators}, we have that $e$  maps integral curves of ${\bf f}$ to solutions of \eqref{coupledoscillators}, exactly  when the conjugacy equation 
$$e' \cdot {\bf f} = {\bf F} \circ e$$ 
holds. If this is the case, then    $\mathbb{T}_{\varepsilon} = e((\mathbb{R}/2\pi\mathbb{Z})^m)$ is the persisting invariant torus, whereas the vector field ${\bf f}$ on $(\mathbb{R}/2\pi\mathbb{Z})^m$ represents the dynamic  on $\mathbb{T}_{\varepsilon}$ in local coordinates, that is, it determines the  reduced phase dynamics.

We will see that the conjugacy equation for $(e, {\bf f})$ translates into a sequence of iterative equations for $(e_1, {\bf f}_1), (e_2, {\bf f}_2), \ldots$. We will show how to solve these iterative equations, which then allows us to compute the expansions for $e$ and ${\bf f}$ to any desired order in the small parameter. 
Because the embedding of the torus $\mathbb{T}_{\varepsilon}$ is not unique, neither are the solutions $(e_j, {\bf f}_j)$ to these iterative equations. We characterize the extent to which  one is free to choose these solutions, and we show how this freedom can be exploited to obtain ${\bf f}_j$ that are in {\it normal form}. This means that ``nonresonant'' terms have been removed from the reduced phase equations to high order.

A crucial requirement for the solvability of the iterative equations is that the torus $\mathbb{T}_0$ is {\it reducible}. Reducibility is a property of the unperturbed dynamics normal to $\mathbb{T}_0$. We shall define it at the hand of an embedding of the so-called {\it fast fibre bundle} of $\mathbb{T}_0$. We call such an embedding a {\it fast fibre map}. The fast fibre map  is an important ingredient of our   method.
An invariant torus for an uncoupled oscillator system is always reducible. We   show in Section \ref{Floquetsection} how, in this case, the fast fibre map can  be obtained from the Floquet decompositions of the fundamental matrix solutions of the periodic orbits $X_j(t)$. 
We remark that by using fast fibre maps, we are able to avoid the use of isochrons \cite{guckenheimer} to characterise the dynamics normal to $\mathbb{T}_0$. Our parametrisation method is therefore  not restricted to the case where the periodic orbits $X_j(t)$ are stable limit cycles---it suffices if they are hyperbolic.
We also stress that our  method is not restricted to weakly coupled oscillator systems: it applies whenever the unperturbed embedded torus $\mathbb{T}_0$ is quasi-periodic,  normally hyperbolic and  reducible.




This paper is organised as follows. In section \ref{sec:iterative} we discuss the conjugacy problem for $(e, {\bf f})$ in more detail, and derive the iterative equations for $(e_j, {\bf f}_j)$. In section \ref{sec:parametrisationsection} we introduce fast fibre maps and use them to define when an embedded (quasi-)periodic torus is reducible. In section \ref{solutionsection} we explain how the fast fibre map can be used to solve the iterative equations for $(e_j, {\bf f}_j)$.  We give formulas for the solutions, and discuss their properties. Section \ref{Floquetsection} shows how to compute the fast fibre map for a coupled oscillator system, treating  the Stuart-Landau oscillator as an example. We finish with an application/illustration of our  method in section \ref{sec:examplesection}, in which we prove that remote synchronisation occurs in a chain of weakly coupled Stuart-Landau oscillators.

\section{An iterative scheme}  \label{sec:iterative}
We start this section with a proof of our earlier claim about the embedding $e_0$. In the formulation of Lemma \ref{basiclemma} below, we use the notation 
\begin{align}\label{e0G0first} \partial_\omega e_0 := e_0'\cdot \omega = \left. \frac{d}{ds}\right|_{s=0} \!\!\!\!\!\!\! e_0(\, \cdot +s \omega) \, 
\end{align} 
for the (directional) derivative of $e_0$ in the direction  of the vector $\omega\in\mathbb{R}^m$. Like $e_0$ itself, $\partial_\omega e_0$  is a smooth map from $(\mathbb{R}/2\pi\mathbb{Z})^m$ to $\mathbb{R}^M$. 
\begin{lem}\label{basiclemma}
 The embedding 
$e_{0}$
  defined in \eqref{unperturbedembedding}
  satisfies the conjugacy equation
$$\partial_{\omega}e_0 \ (= e_0' \cdot \omega) = {\bf F}_0 \circ e_0\, . $$ 
\end{lem}
\begin{proof}
Recall from \eqref{unperturbedembedding} that $(e_0)_j(\phi) = X_j(\omega_j^{-1} \phi_j)$, where $X_j$ is a hyperbolic periodic orbit of $F_j$. It follows that 
\begin{align}\nonumber 
(\partial_{\omega} e_0)_j(\phi) & =
\left. \frac{d}{ds} \right|_{s=0} \!\!\! (e_0)_j(\phi + s \omega) =  
\left. \frac{d}{ds} \right|_{s=0}  \!\!\! X_{j}(\omega_j^{-1}(\phi_j + s\omega_j))  \\ \nonumber
& = \dot X_j(\omega_j^{-1} \phi_j) = F_j(X_j(\omega_j^{-1} \phi_j)) = ({\bf F}_0)_j ( (e_0 (\phi))\, ,  
\end{align}
because $\dot X_j(t) = F_j(X_j(t))$ for all $t \in \mathbb{R}$. 
\end{proof}
\noindent Lemma \ref{basiclemma} implies that 
 $e_0$ sends integral curves of the constant vector field $\omega$ on $(\mathbb{R}/2\pi\mathbb{Z})^m$ to integral curves of the vector field ${\bf F}_0$  given in \eqref{F0def}. Because the integral curves of the ODEs $\dot \phi = \omega$ on $(\mathbb{R}/2\pi\mathbb{Z})^m$ are clearly either periodic or quasi-periodic, we  call $\mathbb{T}_0 = e_0((\mathbb{R}/2\pi\mathbb{Z})^m)$ an embedded {\it (quasi-)periodic torus}. 

At this point we temporarily abandon the setting of coupled oscillators and consider a general ODE $\dot x = {\bf F}_0(x)$ defined by a smooth vector field  
 ${\bf F}_0: \mathbb{R}^M\to\mathbb{R}^M$. That is, we do not assume that this ODE decouples into mutually independent ODEs. However, we will assume throughout this paper that ${\bf F}_0$ possesses a normally hyperbolic periodic or quasi-periodic invariant torus $\mathbb{T}_0$ which admits an embedding
$e_0:( \mathbb{R}/2\pi\mathbb{Z})^m \to \mathbb{R}^M$ 
that semi-conjugates the constant vector field $\omega$
on $( \mathbb{R}/2\pi\mathbb{Z})^m$ to ${\bf F}_0$. In other words, we assume that $e_0$ and ${\bf F}_0$ satisfy 
\begin{align}\label{e0G0}
\partial_{\omega}e_0  = {\bf F}_0 \circ e_0\, ,
\end{align}
just as in Lemma \ref{basiclemma}. We  return to coupled oscillator systems in section \ref{Floquetsection}.

 We now study any smooth perturbation of  ${\bf F}_0$ of the form 
$${\bf F} = {\bf F}(x) = {\bf F}_0(x) + \varepsilon \, {\bf F}_1(x) + \varepsilon^2 \, {\bf F}_2(x) +  \ldots :\ \mathbb{R}^M   \to \mathbb{R}^M\, .$$
F\'enichel's theorem \cite{fenichel1979, wechselberger2020} guarantees that, for $0\leq \varepsilon \ll 1$, the perturbed ODE $\dot x = {\bf F}(x)$ admits an invariant torus $\mathbb{T}_{\varepsilon}$ close to $\mathbb{T}_0$, that depends smoothly on $\varepsilon$.
Our strategy to find  $\mathbb{T}_{\varepsilon}$ will be to search for an embedding 
$e: (\mathbb{R}/2\pi\mathbb{Z})^m \to \mathbb{R}^M$ close to $e_0$, and a reduced vector field ${\bf f}: (\mathbb{R}/2\pi\mathbb{Z})^m \to \mathbb{R}^m$ close to $\omega$ satisfying the 
{\it conjugacy equation}
\begin{align}\label{fullconj}
\mathfrak{C}(e, {\bf f}) := e' \cdot {\bf f} - {\bf F} \circ e = 0  \, .
\end{align}
Any solution $(e, {\bf f})$ to \eqref{fullconj} indeed yields an embedded ${\bf F}$-invariant torus $\mathbb{T}_{\varepsilon}:=e((\mathbb{R}/2\pi\mathbb{Z})^m) \subset \mathbb{R}^M$, as we see from \eqref{fullconj} that at any point $x=e(\phi) \in \mathbb{T}_{\varepsilon}$ the vector ${\bf F}(x)$ lies in the image of the derivative $e'(\phi)$, and is thus tangent to $\mathbb{T}_{\varepsilon}$.  Moreover, $e$ semi-conjugates ${\bf f}$ to ${\bf F}$, that is, ${\bf f}$ is  the restriction of ${\bf F}$ to $\mathbb{T}_{\varepsilon}$ represented in (or ``pulled back to'') the local coordinate chart $(\mathbb{R}/2\pi\mathbb{Z})^m$. 

As explained in the introduction, we try to find solutions to \eqref{fullconj} by making a series expansion ansatz
$$e = e_0 + \varepsilon e_1 +  \varepsilon^2 e_2 + \ldots  \ \mbox{and}\ {\bf f} = \omega + \varepsilon {\bf f}_1 +  \varepsilon^2 {\bf f}_2  + \ldots $$ 
for $e_1, e_2, \ldots : (\mathbb{R}/2\pi\mathbb{Z})^m \to \mathbb{R}^M$ and ${\bf f}_1, {\bf f}_2, \ldots : (\mathbb{R}/2\pi\mathbb{Z})^m \to\mathbb{R}^m$.  
Substitution of this ansatz in   \eqref{fullconj}, and Taylor expansion to $\varepsilon$, yields the following list of recursive equations for the $e_j$ and ${\bf f}_j$:
\begin{align}\label{iterativeeqns}
\begin{array}{ccc}
 ( \partial_{\omega}  -  {\bf F}_0'\circ e_0 )\cdot e_1 + e_0' \cdot {\bf f}_1   = &  \hspace{-5mm}  {\bf F}_1\circ e_0   &   \hspace{-2mm}   =: {\bf G}_1 \\ 
 ( \partial_{\omega}  -  {\bf F}_0'\circ e_0 )\cdot e_2 + e_0' \cdot {\bf f}_2  = & \hspace{-5mm}   {\bf F}_2\circ e_0 + ({\bf F}_1'\circ e_0)\cdot e_1 \\
  &\hspace{-8mm}   + \frac{1}{2}({\bf F}_0'' \circ e_0)(e_1, e_1)  - e_1'\cdot {\bf f}_1 &     \hspace{-2mm}  =:  {\bf G}_2  \\ 
 \vdots  & \vdots & \vdots \\  
 ( \partial_{\omega}  -  {\bf F}_0'\circ e_0 )\cdot e_j + e_0' \cdot {\bf f}_j =   &  \hspace{0mm}  \ldots  &  \hspace{-2mm}   =:   {\bf G}_j  \\
\vdots &  \vdots & \vdots
\end{array}
\end{align}
Here, each ${\bf G}_j: (\mathbb{R}/2\pi\mathbb{Z})^m \to \mathbb{R}^M$ is an ``inhomogeneous term'' that can iteratively be determined and depends on ${\bf F}_1, \ldots, {\bf F}_{j}, {\bf f}_1, \ldots, {\bf f}_{j-1}$ and $e_1, \ldots, e_{j-1}$. Concretely, ${\bf G}_j$ is given by
\begin{align}\label{Gjformula}
\!\!\!\! {\bf G}_{j} \! :=\!\! \left. \frac{1}{j!}\frac{d^{j}}{d\varepsilon^{j}}\right|_{\varepsilon=0} \!\!\!\!\!\!  \begin{array}{c} \!\!\!  
({\bf F}_0 + \varepsilon {\bf F}_1 + \ldots + \varepsilon^j {\bf F}_j) (e_0+ \varepsilon e_1 + \ldots + \varepsilon^{j-1} e_{j-1}) \hspace{1.3cm}    \\ 
- (e_0 + \varepsilon e_1 \ldots + \varepsilon^{j-1} e_{j-1})' \cdot (\omega + \varepsilon {\bf f}_1 +\ldots + \varepsilon^{j-1} {\bf f}_{j-1})\end{array} \!\!\!\!\!\!  .
\end{align}
 Explicit formulas for ${\bf G}_1$ and ${\bf G}_2$ are given in \eqref{iterativeeqns}. Note that equations \eqref{iterativeeqns} are all of the form
\begin{align}\label{formulalinearisationC}
\mathfrak{c}(e_j, {\bf f}_j) = {\bf G}_j \ \mbox{for}\ j = 1,2, \ldots \, ,
\end{align}
in which 
\begin{align}\label{formulalinearisationCexplicit}
\mathfrak{c}(e_j, {\bf f}_j) :=  ( \partial_{\omega}  -  {\bf F}_0'\circ e_0 )\cdot e_j+ e_0' \cdot {\bf f}_j \, 
\end{align}
is the linearisation of the operator $\mathfrak{C}$ defined in \eqref{fullconj} at the point $(e, {\bf f}) = (e_0, \omega)$,  where $\varepsilon=0$. This linearisation $\mathfrak{c}$ is not invertible, but we will see that $\mathfrak{c}$ is surjective under the assumption that  $\mathbb{T}_0$ is reducible. This implies that  equations \eqref{iterativeeqns} can iteratively be solved. 

\begin{remk}
We  think of $\mathfrak{C}$ and $\mathfrak{c}$ as  operators between function spaces. For example, for ${\bf F}_0\in C^{r+1}(\mathbb{R}^M, \mathbb{R}^M), {\bf F}\in C^{r}(\mathbb{R}^M, \mathbb{R}^M)$, and $e_0\in C^{r+1}((\mathbb{R}/2\pi\mathbb{Z})^m, \mathbb{R}^M)$,  
$$\mathfrak{C}, \mathfrak{c}: C^{r+1}((\mathbb{R}/2\pi\mathbb{Z})^m, \mathbb{R}^M) \times C^{r}((\mathbb{R}/2\pi\mathbb{Z})^m, \mathbb{R}^m) \to C^{r}((\mathbb{R}/2\pi\mathbb{Z})^m, \mathbb{R}^M)\, .$$
\end{remk}

\begin{remk}\label{rem:nonunique}
The solutions to equation \eqref{fullconj} are not unique because an invariant torus can be embedded in many different ways. In fact, if $e: (\mathbb{R}/2\pi\mathbb{Z})^m \to \mathbb{R}^M$ is an embedding of $\mathbb{T}_{\varepsilon}$ and $\Psi: (\mathbb{R}/2\pi\mathbb{Z})^m \to (\mathbb{R}/2\pi\mathbb{Z})^m$ is any diffeomorphism of the standard torus, then also $e\circ \Psi$ is an embedding of $\mathbb{T}_{\varepsilon}$. The operator $\mathfrak{C}$ defined in \eqref{fullconj} is thus equivariant under the group of diffeomorphisms of $(\mathbb{R}/2\pi\mathbb{Z})^m$. As a consequence, solutions of \eqref{formulalinearisationC} are  not unique either. 

\end{remk}

\begin{remk}\label{rem:nonunique2}
 For the interested reader we provide additional details on Remark~\ref{rem:nonunique}.  Let us denote by $\Psi^*{\bf{f}}$ the pullback of the vector field ${\bf{f}}$ by $\Psi$ defined by the formula 
 $(\Psi^*{\bf{f}})(\phi) := (\Psi'(\phi))^{-1}\cdot {\bf f}(\Psi(\phi))$ for all $\phi \in (\mathbb{R}/2\pi\mathbb{Z})^m$. 
We claim that 
\begin{align}
\mathfrak{C}(e\circ \Psi, \Psi^*{\bf{f}}) = \mathfrak{C}(e, f)\circ \Psi\, .
\end{align}
This follows from a straightforward calculation. Indeed,  
\begin{align}
\mathfrak{C}(e\circ \Psi, \Psi^*{\bf{f}})(\phi) &= e'(\Psi(\phi))\cdot\Psi'(\phi)\cdot(\Psi'(\phi))^{-1}\cdot{\bf f}(\Psi(\phi)) - {\bf F}((e \circ \Psi)(\phi)) \nonumber \\ \nonumber
&= e'(\Psi(\phi))\cdot{\bf f}(\Psi(\phi)) - ({\bf F} \circ e)(\Psi(\phi)) = \mathfrak{C}(e, f)(\Psi(\phi))\, .
\end{align}
As we may view vector fields as infinitesimal diffeomorphisms, this allows us to find many elements in the kernel of $\mathfrak{c}$. Namely, if $X$ is any vector field on $(\mathbb{R}/2\pi\mathbb{Z})^m$ with corresponding flow $\varphi_t$, then 
\begin{align}\label{fancykernelformula}
\left.\frac{d}{dt}\right|_{t=0}\hspace{-12pt}(e_0 \circ \varphi_t, \varphi_t^* \omega) = (e_0' \cdot X, [X, \omega]) \in \ker \mathfrak{c}.
\end{align}
Here $ [X, \omega] = -X'\cdot\omega = -\partial_\omega X$ denotes the Lie bracket between $X$ and $\omega$.

Formula \eqref{fancykernelformula} may also be  verified directly. Differentiating the identity
\begin{align}
\mathfrak{C}(e_0, \omega)(\phi) = e_0'(\phi) \cdot \omega - ({\bf F}_0 \circ e_0)(\phi) = 0
\end{align}
at any $\phi$, in the direction of any vector $u$, we first of all find that
\begin{align}\label{Eremk1}
e_0''(\phi) (\omega, u) - ({\bf F}_0' \circ e_0)(\phi) \cdot e_0'(\phi) \cdot u = 0\, .
\end{align}
From this we see that indeed
\begin{align} \nonumber
\mathfrak{c}(e_0' \cdot X, [X, \omega])   &=   (\partial_\omega - {\bf F}_0' \circ e_0)\cdot e_0' \cdot X - e'_0 \cdot \partial_\omega X \\ \nonumber
&=   e_0'' (\omega, X) + e_0' \cdot \partial_\omega X  - ({\bf F}_0' \circ e_0) \cdot e_0' \cdot X - e'_0 \cdot \partial_\omega X \\ \nonumber
&=   e_0'' (\omega, X)  - ({\bf F}_0' \circ e_0) \cdot e_0' \cdot X = 0 \, ,
\end{align}
where the last step follows from equation \eqref{Eremk1}. 

\end{remk}

\section{Reducibility and the fast fibre map} \label{sec:parametrisationsection} 
As was indicated in Remarks \ref{rem:nonunique} and \ref{rem:nonunique2}, the solutions to the iterative equations $\mathfrak{c}(e_j, {\bf f}_j) = {\bf G}_j $ are not unique. However, we show in section \ref{solutionsection}  that solutions can be found if we assume that the unperturbed  torus $\mathbb{T}_0$ is  reducible. We define this concept by means of a parametrisation of the linearised dynamics of ${\bf F}_0$ normal to $\mathbb{T}_0$. 
But we start with the observation that the linearised dynamics tangent to $\mathbb{T}_0$ is trivial.
Recall that if $e_0: (\mathbb{R}/2\pi\mathbb{Z})^m \to \mathbb{R}^M$ is an embedding of $\mathbb{T}_0 \subset \mathbb{R}^M$, then the {\it tangent map}
${\bf T}e_0 :(\mathbb{R}/2\pi\mathbb{Z})^m \times \mathbb{R}^m \to \mathbb{R}^M\times \mathbb{R}^M$ defined by 
\begin{align}\label{def:Te0}
{\bf T}e_0(\phi, u) = (e_0(\phi), e_0'(\phi)\cdot u)
\end{align}
is an embedding as well. Its image is the tangent bundle 
${\bf T}\mathbb{T}_0 \subset \mathbb{R}^M \times \mathbb{R}^M$. 
 
\begin{lem} \label{tangentembedding}
Assume that the embedding $e_0: (\mathbb{R}/2\pi\mathbb{Z})^m \to\mathbb{R}^M$ semi-conjugates the constant vector field $\omega \in \mathbb{R}^m$ on $(\mathbb{R}/2\pi\mathbb{Z})^m$ to the vector field ${\bf F}_0$ on $\mathbb{R}^M$. 
Then ${\bf T}e_0$ sends solution curves of the system of ODEs 
$$\dot \phi = \omega\, , \ \dot u = 0\ \mbox{on} \ (\mathbb{R}/2\pi\mathbb{Z})^m \times \mathbb{R}^m$$  to integral curves of the tangent vector field ${\bf T}{\bf F}_0$ on $\mathbb{R}^M\times\mathbb{R}^M$ defined by 
\begin{align} \label{def:TF0}
{\bf T}{\bf F}_0(x,v) := ({\bf F}_0(x), {\bf F}_0'(x)\cdot v)\, .
\end{align}
\end{lem}
\begin{proof}
Our assumption simply means that 
$\partial_{\omega} e_0 = {\bf F}_0 \circ e_0$. As we already observed in \eqref{Eremk1}, differentiation of this identity at a point $\phi \in (\mathbb{R}/2\pi\mathbb{Z})^m$ in the direction of a vector $u \in \mathbb{R}^m$ yields that  
\begin{align}\nonumber 
 e_0''(\phi)  (u, 
 \omega) = {\bf F}_0'(e_0(\phi))\cdot e_0'(\phi) \cdot u\, .
\end{align}
From this it follows that 
\begin{align}\nonumber 
 ({\bf T}e_0)'(\phi,u)\cdot ( \omega, 0)  = &  \left. \frac{d}{ds}\right|_{s=0} \!\!\! \!\! {\bf T}e_0(\phi+s \omega, u) \\ \nonumber 
 = & \left. \frac{d}{ds}\right|_{s=0} \!\!\!\!\! \left( e_0(\phi+s\omega), e_0'( \phi+s\omega) \cdot u \right)  \\
\nonumber 
 = &  \left( (\partial_{\omega} e_0)(\phi),  e_0''(\phi) (u, 
 \omega)  \right)   
\\
 \nonumber
= & \left( {\bf F}_0(e_0(\phi)), {\bf F}_0'(e_0(\phi)) \cdot e_0'(\phi)\cdot u \right)   
=  {\bf T}{\bf F}_0 ( {\bf T}e_0(\phi, u))\, .
\end{align}
In the last equality we used Definitions \eqref{def:Te0} and \eqref{def:TF0}. 
\end{proof}
\noindent Lemma \ref{tangentembedding} shows that ${\bf T}e_0$ trivialises the linearised dynamics of ${\bf F}_0$ in the direction tangent to $\mathbb{T}_0$. In what follows, we assume that something similar  happens in the direction normal to $\mathbb{T}_0$, that is, we assume that $\mathbb{T}_0$ is reducible. We define this concept now.

\begin{defi}\label{reducibledefi}
Assume that the embedding $e_0: (\mathbb{R}/2\pi\mathbb{Z})^m \to\mathbb{R}^M$ semi-conjugates the constant vector field $\omega \in \mathbb{R}^m$ on $(\mathbb{R}/2\pi\mathbb{Z})^m$ to the vector field ${\bf F}_0$ on $\mathbb{R}^M$. 
We say that the (quasi-)periodic invariant torus $\mathbb{T}_0 = e_0( (\mathbb{R}/2\pi\mathbb{Z})^m)$ is {\it reducible} if there is a map 
 ${\bf N}e_0 \! : \!(\mathbb{R}/2\pi\mathbb{Z})^m \! \times \! \mathbb{R}^{M-m} \to\mathbb{R}^M \! \times \! \mathbb{R}^M$ of the form 
 \begin{align}
\label{NGammadefgeneral}
{\bf N}e_0(\phi, u) := (e_0(\phi), N(\phi)\cdot u) \,  , 
\end{align}
for a smooth family of  linear maps
$$N: (\mathbb{R}/2\pi\mathbb{Z})^m \to \mathcal{L}(\mathbb{R}^{M-m}, \mathbb{R}^M)\, , $$
with the following two properties:
\begin{itemize}
\item[{\it i)}]  
  ${\bf N}e_0$ is transverse to ${\bf T}e_0$. 
 By this we mean that 
\begin{align}\label{transversality}
\mathbb{R}^M = {\rm im}\, e_0'(\phi) \oplus {\rm im}\, N(\phi)\ \mbox{for every}\ \phi \in (\mathbb{R}/2\pi\mathbb{Z})^m\, .
\end{align}
In particular, every $N(\phi)$ is injective.
\item[{\it ii)}]  There is a linear map $L: \mathbb{R}^{M-m} \to  \mathbb{R}^{M-m}$ such that ${\bf N}e_0$
sends solution curves of the system of ODEs
$$\dot \phi = \omega\, ,\, \dot u = L \cdot u \ \mbox{defined on}\ (\mathbb{R}/2\pi\mathbb{Z})^m \times \mathbb{R}^{M-m}$$
to integral curves of the tangent vector field ${\bf T}{\bf F}_0$ on $\mathbb{R}^M\times \mathbb{R}^M$.
\end{itemize}
When $\mathbb{T}_0$ is reducible, the matrix $L$ is called a {\it Floquet matrix} for $\mathbb{T}_0$, and its eigenvalues  the {\it Floquet exponents} of $\mathbb{T}_0$. 

If $L$ is hyperbolic (no Floquet exponents lie on the imaginary axis) then $\mathbb{T}_0$ is normally hyperbolic, and we call ${\bf N}e_0$ a {\it fast fibre map} for $\mathbb{T}_0$. Its image $${\bf N}\mathbb{T}_0 := {\bf N}e_0 ((\mathbb{R}/2\pi\mathbb{Z})^m \times \mathbb{R}^{M-m}) \subset \mathbb{R}^M\times\mathbb{R}^M\, $$
is then called the {\it fast fibre bundle} of $\mathbb{T}_0$.
\end{defi}
\noindent We note that the map ${\bf N}e_0$ appearing in Definition \ref{reducibledefi} is an embedding because $e_0$ is an embedding and the linear maps $N(\phi)$ are all injective. Therefore its image  ${\bf N}\mathbb{T}_0$  
 is a smooth $M$-dimensional manifold.  Condition {\it i)} ensures that ${\bf N}\mathbb{T}_0$ is in fact a normal bundle for $\mathbb{T}_0$.   
 
We finish this section with an alternative characterisation of property $\it ii)$ in Definition \ref{reducibledefi}. 
\begin{lem}\label{variationalparametrisation}
Assume that the embedding $e_0:  (\mathbb{R}/2\pi\mathbb{Z})^m \to\mathbb{R}^M$ semi-conjugates the constant vector field $\omega$  to the vector field ${\bf F}_0$. Let $L: \mathbb{R}^{M-m} \to  \mathbb{R}^{M-m}$ be a linear map, and 
 let ${\bf N}e_0$ be a map of the form \eqref{NGammadefgeneral} for a smooth family of linear maps $N: (\mathbb{R}/2\pi\mathbb{Z})^m \to \mathcal{L}(\mathbb{R}^{M-m}, \mathbb{R}^M)$.  
The following are equivalent:
\begin{itemize}
\item[{\it i)}] ${\bf N}e_0$ sends solution curves of the system of ODEs
$$\dot \phi = \omega\, ,\, \dot u = L \cdot u \ \mbox{defined on}\ (\mathbb{R}/2\pi\mathbb{Z})^m \times \mathbb{R}^{M-m}$$
to integral curves of the tangent vector field ${\bf T}{\bf F}_0$ on $\mathbb{R}^M\times \mathbb{R}^M$;
\item[{\it ii)}]  $N=N(\phi)$ satisfies the partial differential equation
\begin{align}\label{N0M0}
\partial_{\omega} N + N \! \cdot \! L= ({\bf F}_0' \circ e_0) \! \cdot \! N\ \mbox{on}\ (\mathbb{R}/2\pi\mathbb{Z})^m\, .
\end{align}
\end{itemize} 
  \end{lem}
\begin{proof}

It holds that
\begin{align}
({\bf N}e_0)'(\phi, u) \cdot (\omega, L\cdot u ) & = \left. \frac{d}{ds}\right|_{s = 0} \!\!\!\!\! \left(e_0(\phi + s \omega), N(\phi + s \omega) \cdot (u+s L\cdot u ) \right)
\nonumber  \\ \nonumber 
 & = ((\partial_{\omega} e_0)(\phi), \partial_{\omega} N(\phi) \cdot u + N(\phi)\cdot L\cdot u)   \, .
 \end{align} 
 At the same time,
 $${\bf T}{\bf F}_0({\bf N}e_0(\phi, u)) = ({\bf F}_0(e_0(\phi)), {\bf F}_0'(e_0(\phi))\cdot N(\phi)\cdot u)\, .$$
 It holds that  
 $\partial_{\omega} e_0 = {\bf F}_0\circ e_0$ by assumption, so  the first components of these two expressions are equal. The conclusion of the lemma therefore follows from comparing  the second components.
\end{proof}
\begin{remk}
Reducibility of a (quasi-)periodic invariant torus of an arbitrary vector field ${\bf F}_0$ can only be quaranteed under strong conditions, e.g., that ${\bf F}_0$ is Hamiltonian \cite{KAMwithout},  or that the frequency vector $\omega$ satisfies certain Diophantine inequalities  \cite{JohnsonSell}. We do not assume such conditions here. Even the question whether reducibility is preserved under perturbation is subtle \cite{JorbaSimo}. 

However, hyperbolic periodic orbits (which are one-dimensional normally hyperbolic invariant tori) are always reducible (at least if we allow the matrix $L$ to be complex, see Section \ref{Floquetsection}). This relatively well-known fact is a consequence of Floquet's theorem \cite{floquet}, as  we show in Theorem \ref{floquetreducible}. The (quasi-)periodic  torus occurring in an uncoupled oscillator system such as \eqref{coupledoscillators} is a product of hyperbolic periodic orbits, and is therefore reducible as well, see Lemma \ref{productreduciblelemma}.  
\end{remk}



\section{Solving the iterative equations} \label{solutionsection}
We now return to solving the iterative  equations \eqref{iterativeeqns}, assuming from here on out that $\mathbb{T}_0$ is an embedded (quasi-)periodic reducible and normally hyperbolic invariant torus for ${\bf F}_0$. The main result of this section can be summarised (at this point still somewhat imprecisely) as follows.
\begin{thr}\label{sloppythm}
Assume that $\mathbb{T}_0 = e_0((\mathbb{R}/2\pi\mathbb{Z})^m) \subset \mathbb{R}^M$ is a smooth embedded (quasi-)periodic reducible normally hyperbolic invariant torus for ${\bf F}_0$. Then 
\begin{itemize}
\item[{\it i)}] there are smooth solutions $(e_j, {\bf f}_j)$ to the iterative equations $\mathfrak{c}(e_j, {\bf f}_j)={\bf G}_j$ for every $j\in \mathbb{N}$, for which we provide explicit formulas in this section;
\item[{\it ii)}] the component of each $e_j$ tangential to $\mathbb{T}_0$ can be chosen freely, but every such choice for $e_1, \ldots, e_{j-1}$ uniquely determines the component of $e_j$ normal to $\mathbb{T}_0$ (see Theorem \ref{solutionlemma});
\item[{\it iii)}] the tangential component of $e_j$ can be chosen in such a way that ${\bf f}_j$ is in ``normal form'' to arbitrarily high order in its Fourier expansion. We say that ${\bf f}_j$ is in normal form if it is a sum of  ``resonant terms'' only (see Corollary \ref{normalformcorollary}).
\end{itemize} 
\end{thr}
\noindent The precise meaning of the statements in this theorem will be made clear below. Theorem \ref{sloppythm} follows directly from the results presented in this section. 

To prove the theorem, recall that (because $\mathbb{T}_0$ is reducible) we have at our disposal a fast fibre map ${\bf N}e_0$ for $\mathbb{T}_0$, defined by a family of injective matrices $N=N(\phi)$ that satisfies $\mathbb{R}^M = {\rm im}\, e_0'(\phi) \oplus {\rm im}\, N(\phi)$ for every $\phi \in (\mathbb{R}/2\pi\mathbb{Z})^m$. This enables us to make the ansatz
\begin{align}\label{Ansatz}
\underbrace{e_j(\phi)}_{\footnotesize \begin{array}{c} \in \\  \mathbb{R}^M \end{array}} \ \  = \!\! \underbrace{e_0'(\phi)}_{\footnotesize \begin{array}{c} \in \\  \mathcal{L}(\mathbb{R}^m, \mathbb{R}^M) \end{array}}
 \!\!\!\!\! \cdot  \  
 \underbrace{{\bf g}_j(\phi)}_{\footnotesize \begin{array}{c} \in \\  \mathbb{R}^m \end{array} }\ \  + \!\!\!\!\!
 \underbrace{N(\phi)}_{\footnotesize \begin{array}{c} \in \\   \mathcal{L}(\mathbb{R}^{M-m}, \mathbb{R}^M) \end{array}  }
 \!\!\!\!\!\!\! \cdot \ 
 \underbrace{{\bf h}_j(\phi)}_{\footnotesize \begin{array}{c} \in \\  \mathbb{R}^{M-m} \end{array}} \, ,
\end{align}
for (unknown) smooth functions ${\bf g}_j: (\mathbb{R}/2\pi\mathbb{Z})^m \to \mathbb{R}^m$ and ${\bf h}_j: (\mathbb{R}/2\pi\mathbb{Z})^m \to \mathbb{R}^{M-m}$. This ansatz decomposes  $e_j$ into  components in the direction of the tangent bundle ${\bf T}e_0$ and the fast fibre bundle ${\bf N}e_0$.
\begin{lem} \label{ansatzlemma}
The ansatz \eqref{Ansatz} transforms equation \eqref{formulalinearisationC} into 
\begin{align}\label{newequation}
\mathfrak{c}(e_j, {\bf f}_j) = e_0' \cdot \left( \partial_{\omega} {\bf g}_j + {\bf f}_j \right) + N \cdot(\partial_{\omega} - L)({\bf h}_j) = {\bf G}_j\, .
\end{align}
\end{lem}
\begin{proof}
We use our definitions, and  results  derived above, to compute:
\begin{align}\nonumber
{\bf G}_j = \mathfrak{c}(e_j, {\bf f}_j) = &  \ ( \partial_{\omega}  -  {\bf F}_0'\circ e_0 )\cdot e_j+ e_0' \cdot {\bf f}_j 
\\ \nonumber
 =  &\  ( \partial_{\omega}  -  {\bf F}_0' \circ e_0 )\cdot \left( e_0' \cdot {\bf g}_j + N\cdot {\bf h}_j \right) + e_0' \cdot {\bf f}_j
 \\ \nonumber
 = & \  e_0''({\bf g}_j, \omega)    + e_0' \cdot \partial_{\omega}{\bf g}_j    +  \partial_{\omega}N \cdot {\bf h}_j  + N \cdot \partial_{\omega}{\bf h}_j 
\\ \nonumber 
& \   - ({\bf F}_0'. \circ e_0) \cdot  e_0' \cdot {\bf g}_j  -   ({\bf F_0}' \circ e_0)\cdot N \cdot {\bf h}_j  + e_0' \cdot {\bf f}_j
\\ \nonumber 
 = & \   \underbrace{e_0''({\bf g}_j, \omega)  -   ({\bf F}_0'. \circ e_0) \cdot  e_0' \cdot {\bf g}_j }_{=0} + e_0' \cdot \partial_{\omega}{\bf g}_j + e_0' \cdot {\bf f}_j
\\ \nonumber 
& \ + N \cdot \partial_{\omega}{\bf h}_j + \underbrace{\partial_{\omega}N \cdot {\bf h}_j  -   ({\bf F_0}' \circ e_0)\cdot N \cdot {\bf h}_j}_{=-N \cdot L \cdot {\bf h}_j}
\\ \nonumber 
 = & \ e_0' \cdot \left( \partial_{\omega}{\bf g}_j + {\bf f}_j \right) + N \cdot \left(\partial_{\omega} - L \right) \cdot {\bf h}_j\, . 
\end{align}
We clarify these equalities below:
\begin{itemize}
\item[1.] The first equality is \eqref{formulalinearisationC}; 
\item[2.] In the second equality, we used \eqref{formulalinearisationCexplicit};
\item[3.] The third equality is our ansatz \eqref{Ansatz};
\item[4.] The fourth equality follows from the product rule (applied twice); 
\item[5.] In the fifth equality, the terms in the sum were re-ordered;
\item[6.] The final equality follows from \eqref{Eremk1} and 
\eqref{N0M0}.
\end{itemize}
This proves the lemma.
\end{proof}
\noindent Lemma \ref{ansatzlemma} allows us to solve equation \eqref{newequation} by splitting it into a component along the tangent bundle ${\bf T}\mathbb{T}_0$ and a component along the fast fibre bundle ${\bf N}\mathbb{T}_0$ of $\mathbb{T}_0$. In what follows we  denote by 
$$\pi: (\mathbb{R}/2\pi\mathbb{Z})^m \to \mathcal{L}(\mathbb{R}^M, \mathbb{R}^M)$$
the family of projections onto the tangent bundle  along the fast fibre bundle. That is, each $\pi(\phi): \mathbb{R}^M\to\mathbb{R}^M$ is the unique projection that satisfies
$$\pi(\phi)\cdot e_0'(\phi) = e_0'(\phi)\ \mbox{and}\ \pi(\phi)\cdot N(\phi) = 0\, .$$
Proposition \ref{projectionprop} below provides an explicit formula for $\pi(\phi)$. It is clear from this formula that $\pi$ depends smoothly on the base point $\phi \in (\mathbb{R}/2\pi\mathbb{Z})^m$.

Applying $\pi$ and $1-\pi$ to \eqref{newequation} produces, respectively, 
 \begin{align}\nonumber 
& e_0' \cdot ( \partial_{\omega} {\bf g}_j + {\bf f}_j ) =  \pi \cdot {\bf G}_j \, , \\  \nonumber 
& N \cdot (\partial_{\omega} - L)({\bf h}_j) = (1-\pi)\cdot {\bf G}_j\, .
\end{align}
Because $e_0'(\phi)$ and $N(\phi)$ are injective, these equations are  equivalent to 
\begin{align}\label{tworeducedequations} 
\begin{array}{rll}
 \partial_{\omega} {\bf g}_j + {\bf f}_j   = &  (e_0')^{+} \cdot \pi \cdot {\bf G}_j & =: U_j  \, , \\  
 (\partial_{\omega} - L)({\bf h}_j)  = & N^{+}\cdot(1-\pi)\cdot {\bf G}_j & =: V_j \, .
\end{array}
\end{align}
Here, $A^+ := (A^TA)^{-1}A^T$ denotes the Moore-Penrose pseudo-inverse, which is well-defined for an injective linear map $A$. Clearly, $(e_0')^{+}$ and $N^+$ depend smoothly on $\phi\in(\mathbb{R}/2\pi\mathbb{Z})^m$. We give these equations a special name.
\begin{defi}
We call the first equation in \eqref{tworeducedequations}, 
\begin{align} \label{tworeducedequations1}
 \partial_{\omega} {\bf g}_j + {\bf f}_j  = U_j\, ,
 \end{align}
 the $j$-th {\it tangential homological equation}. The second equation in \eqref{tworeducedequations},  
\begin{align} \label{tworeducedequations2}
 (\partial_{\omega} - L)({\bf h}_j)  = V_j\, , 
 \end{align}
is called the $j$-th {\it normal homological equation}. 
\end{defi}
\begin{remk}
To recap, we note that \eqref{tworeducedequations1} and \eqref{tworeducedequations2} are inhomogeneous linear equations for the three unknown smooth functions ${\bf f}_j, {\bf g}_j, {\bf h}_j$ and with the  inhomogeneous right hand sides $U_j, V_j$. 
The domains and co-domains of these functions are given by
 \begin{align} \nonumber
 {\bf f}_j, {\bf g}_j, U_j : (\mathbb{R}/2\pi\mathbb{Z})^m \to \mathbb{R}^m \ \mbox{and} \ {\bf h}_j, V_j: (\mathbb{R}/2\pi\mathbb{Z})^m \to \mathbb{R}^{M-m} \, .
 \end{align}
  \end{remk}
\noindent The following theorem shows how the homological equations can be solved. Explicit expressions for the Fourier series of the solutions are given in formulas \eqref{fintermsofuandx} and \eqref{YintermsofV}, that appear in the proof of the theorem.
\begin{thr}\label{solutionlemma}

 For any smooth functions ${\bf g}_j, U_j: (\mathbb{R}/2\pi\mathbb{Z})^m \to \mathbb{R}^m$ and $V_j: (\mathbb{R}/2\pi\mathbb{Z})^m \to \mathbb{R}^{M-m}$, there are unique smooth functions  ${\bf f}_j: (\mathbb{R}/2\pi\mathbb{Z})^m \to \mathbb{R}^m$ and ${\bf h}_j:(\mathbb{R}/2\pi\mathbb{Z})^m \to \mathbb{R}^{M-m}$ that solve \eqref{tworeducedequations1} and \eqref{tworeducedequations2}. 
\end{thr}
\begin{proof}
The tangential homological equation \eqref{tworeducedequations1} can be rewritten as
$${\bf f}_j = U_j - \partial_{\omega} {\bf g}_j\, .$$
This shows that for any smooth ${\bf g}_j$ and $U_j$ there exists a unique solution ${\bf f}_j$. However, in view of Corollary \ref{normalformcorollary} below, we would also like a formula for the solution of the tangential homological equation in the form of a Fourier series. To this end, we expand $U_j$ and ${\bf g}_j$ in Fourier series as
$$U_j(\phi) = \sum_{k\in \mathbb{Z}^m} U_{j,k} e^{i\langle k, \phi\rangle} \ \mbox{and}\ {\bf g}_j(\phi) = \sum_{k\in \mathbb{Z}^m} g_{j,k} e^{i\langle k, \phi\rangle}\, .$$
We use the notation 
$$\langle k, \phi\rangle := k_1 \phi_1 + \ldots + k_m \phi_m\, $$
for what is often called the $k$-th {\it combination angle}. 
Note that the Fourier coefficients
$ U_{j,k}, g_{j,k} \in \mathbb{C}^{m}$ are complex vectors satisfying $U_{j,-k} = \overline{U}_{j,k}$ and $g_{j,-k} = \overline{g}_{j,k}$, because $U_j$ and ${\bf g}_j$ are real-valued. 
We similarly expand ${\bf f}_j$ in a Fourier series by making the solution ansatz 
$${\bf f}_j(\phi) = \sum_{k\in \mathbb{Z}^m} f_{j,k} e^{i\langle k, \phi\rangle}, $$
with $f_{j,k}\in \mathbb{C}^m$. 
In terms of these Fourier series, equation \eqref{tworeducedequations1} becomes
$$\sum_{k\in \mathbb{Z}^m} (i\langle \omega,k\rangle g_ {j,k} + f_{j,k})  e^{i\langle k, \phi\rangle} = \sum_{k\in \mathbb{Z}^m}  U_{j,k}  e^{i\langle k, \phi\rangle}\, ,$$
or, equivalently,
$$  i\langle \omega,k\rangle g_ {j,k} + f_{j,k} = U_{j,k}\ \mbox{for all}\ k\in \mathbb{Z}^m\, . $$
This shows that for any choice of Fourier coefficients $U_{j,k}$ for $U_j$ and $g_{j,k}$ for ${\bf g}_j$ there are unique Fourier coefficients $f_{j,k}$ for the solution ${\bf f}_j$ to the tangential homological equation. These coefficients  are given by 
\begin{align}\label{fintermsofuandx}
f_{j,k} = U_{j,k} -  i\langle \omega,k\rangle g_ {j,k} \ \mbox{for all}\ k\in \mathbb{Z}^m\, .
\end{align}
It is clear from this equation that  $f_{j,-k} = \overline{f}_{j,k}$ so that ${\bf f}_j$ is real-valued.

We proceed to solve the normal homological equation   \eqref{tworeducedequations2}.  We again use Fourier series, and thus we expand ${\bf h}_j$  and $V_j$  as 
\begin{align}\label{VYexpansion}
{\bf h}_{j}(\phi)  = \sum_{k\in \mathbb{Z}^m} h_{j,k} e^{i\langle k, \phi\rangle}  \ \mbox{and}\  
V_{j}(\phi)    = \sum_{k\in \mathbb{Z}^m} V_{j,k} e^{i\langle k, \phi\rangle } \, ,
\end{align}
for $h_{j,k}, V_{j,k}\in \mathbb{C}^{M-m}$ satisfying $V_{j,-k} = \overline{V}_{j,k}$. 
Substitution of \eqref{VYexpansion}  into \eqref{tworeducedequations2} produces
$$\sum_{k\in \mathbb{Z}^m} (i\langle \omega, k \rangle - L ) { h}_{j,k} e^{i\langle k, \phi\rangle} = \sum_{k\in \mathbb{Z}^m} V_{j,k} e^{i\langle k, \phi\rangle}\, , $$
so that we obtain the  equations
\begin{align}
\label{iomegaLequation}
( i\langle \omega,k\rangle - L ) \, { h}_{j,k} = V_{j,k} \ \mbox{for all}\ k\in \mathbb{Z}^m\, .
\end{align}
Because  $L$ has no eigenvalues on the imaginary axis, the  matrix $i\langle \omega,k\rangle - L$ is invertible. Each of the equations in  \eqref{iomegaLequation}  therefore possesses a unique solution, which is given by 
\begin{equation}\label{YintermsofV}
{ h}_{j,k} = ( i\langle \omega,k\rangle - L )^{-1}  V_{j,k} \, . 
\end{equation}
Because the matrix $L$ is real, it follows that $h_{j,-k} = \overline{h}_{j,k}$. This proves the theorem.
\end{proof}

\begin{remk} Formulas \eqref{fintermsofuandx} and \eqref{YintermsofV} allow us to estimate the smoothness of the solutions ${\bf f}_j$ and ${\bf h}_j$ to equations \eqref{tworeducedequations1}, \eqref{tworeducedequations2} in terms of the smoothness of ${\bf g}_j, U_j$ and $V_j$. 
To see this, let ${\bf A}: (\mathbb{R}/2\pi\mathbb{Z})^m \to \mathbb{C}^p$ be a function with Fourier series $${\bf A}(\phi) = \sum_{k\in \mathbb{Z}^m} A_ke^{i\langle k, \phi\rangle}\, .$$

For  $k\in \mathbb{Z}^m$, define $|k|:= \left( |k_1|^2+\ldots+|k_m|^2 \right)^{\frac{1}{2}}$, and let $W_{|k|} \in \mathbb{R}_{>0}$ be weights satisfying $W_{|k|}\to\infty$ as $|k|\to\infty$.
When $||\cdot||$ is any norm on $\mathbb{C}^p$, 
then 
$$|| {\bf A}||_W := \left( \sum_{k\in \mathbb{Z}^m} ||A_k||^2 W_{|k|}^2 \right)^{\frac{1}{2}}$$
defines a norm of ${\bf A}$ that measures the growth of its Fourier coefficients. For example, when $W_{|k|}=(1+|k|^2)^{s/2}$ for some $s >  0$, then it is a Sobolev norm. 
It follows directly from \eqref{fintermsofuandx} that $|| {\bf f}_j ||_W \leq ||U_j||_W +  || \partial_{\omega} {\bf g}_j||_W$, which shows that ${\bf f}_j$ is at least as smooth as $U_j$ and $\partial_{\omega} {\bf g}_j$.

To find a similar bound for $||{\bf h}_j||_W$, note that the hyperbolicity of $L$ implies that the function $\lambda \mapsto ||(i\lambda - L)^{-1}||_{\rm op}$  on $\mathbb{R}$,
that assigns to $\lambda$ the operator norm of $(i\lambda - L)^{-1}$, is  well-defined, and therefore also continuous. It converges to $0$ as $\lambda \to \pm \infty$. Hence it is uniformly bounded in $\lambda$. In particular, 
$$||(i\langle k, \omega \rangle - L)^{-1}||_{\rm op} \leq C_L := \max_{\lambda\in\mathbb{R}} ||(i\lambda - L)^{-1}||_{\rm op}\, . $$  
It thus follows from \eqref{YintermsofV} that 
$$||{\bf h}_j||_W \leq C_L ||V_j||_W\, .$$
This means that ${\bf h}_j$ is at least as smooth as $V_j$.
\end{remk}

\noindent Theorem \ref{solutionlemma} shows that one can choose ${\bf g}_j$ (and thus the component of $e_j$ tangent to $\mathbb{T}_0$) freely when  solving the homological equations \eqref{tworeducedequations1} and \eqref{tworeducedequations2}. This reflects the fact that the embedding of $\mathbb{T}_{\varepsilon}$ is not unique. Corollary \ref{normalformcorollary} below states that it is possible to choose ${\bf g}_j$ in such a way that ${\bf f}_j$ is in ``normal form''. We first define this concept.
\begin{defi} 
Let $${\bf f} = \omega + \varepsilon {\bf f}_1 + \varepsilon^2 {\bf f}_2 + \ldots : (\mathbb{R}/2\pi\mathbb{Z})^m \to \mathbb{R}^m$$ be an asymptotic expansion of a vector field on $(\mathbb{R}/2\pi\mathbb{Z})^m$. Assume that  the Fourier series of ${\bf f}_j$ is given by 
$${\bf f}_j(\phi) = \sum_{k\in \mathbb{Z}^m} f_{j, k}e^{i\langle k, \phi\rangle} \ \mbox{for certain}\ f_{j,k}\in \mathbb{C}^m\, .$$
For  $k\in \mathbb{Z}^m$, denote $|k|= \left( |k_1|^2+\ldots+|k_m|^2 \right)^{\frac{1}{2}}$ as before.
We say that ${\bf f}_j$ 
is {\it in normal form to order $K\in \mathbb{N}\cup\{\infty\}$ in its Fourier expansion} if $$f_{j,k} = 0 \ \mbox{for all}\ k\in\mathbb{Z}^m \ \mbox{with} \ \langle \omega, k \rangle \neq 0\ \mbox{and} \ |k|\leq K\, .$$ 
\end{defi}
\begin{remk} We remark that ${\bf f}_j$ is in normal form to order $K$ in its Fourier expansion, if and only if its truncated Fourier series $${\bf f}_j^K(\phi) := \sum_{|k|\leq K} f_{j,k} e^{i\langle k, \phi\rangle}$$ depends only on so-called {\it resonant combination angles}. A combination angle $\langle k, \phi\rangle$ is called resonant when $\langle k, \omega\rangle =0$. 
\end{remk}
\noindent The following result shows that we can arrange for the reduced phase vector field to be in normal form to arbitrarily high-order in its Fourier expansion.
\begin{cor}\label{normalformcorollary}
For any (finite) $K\in \mathbb{N}$ the function ${\bf g}_j$ can be chosen in such a way that the solution ${\bf f}_j$ to the tangential homological equation 
$$\partial_{\omega}{\bf g}_j + {\bf f}_j  = U_j$$
is in normal form 
to order $K$ in its Fourier expansion.
\end{cor}
\begin{proof}
Recall that the tangential homological equation reduces to the equations
\begin{align}\label{tangentialfourier}
i\langle \omega,k\rangle g_ {j,k} + f_{j,k} = U_{j,k} 
\end{align}
for the Fourier coefficients of ${\bf f}_j$, ${\bf g}_j$ and $U_j$---see \eqref{fintermsofuandx}.  Given $K\in \mathbb{N}$,  choose 
\begin{align}\label{XintermsofU}
\begin{array}{ll} 
g_{j,k}  = \frac{U_{j,k} }{i \langle k, \omega \rangle} & \mbox{when}\ \langle k, \omega \rangle \neq 0 \ \mbox{and}\ |k|\leq K, \\
g_{j,k}  = 0  &   \mbox{when} \ \langle k, \omega \rangle = 0 \ \mbox{or} \ |k|> K.
\end{array}
\end{align} 
The (unique) solutions to \eqref{tangentialfourier} are then given by
\begin{align}
\begin{array}{ll} 
f_{j,k}=0 & \mbox{when}\ \langle k, \omega \rangle \neq 0 \ \mbox{and}\ |k|\leq K, \\
f_{j,k}=U_{j,k} &  \mbox{when} \ \langle k, \omega \rangle = 0 \ \mbox{or} \ |k|> K.
\end{array}
\end{align}
With these choices, ${\bf g}_j$ is a smooth function, as its Fourier expansion is finite. It is also clear that ${\bf f}_j$ is in normal form to order $K$ in its Fourier expansion.
\end{proof}




\begin{remk}
Recall that the flow of the ODE $\dot \phi = \omega$  on $(\mathbb{R}/2\pi\mathbb{Z})^m$ is periodic or quasi-periodic and given by the formula $\phi \mapsto \phi + \omega t \ \mbox{mod}\, (2\pi \mathbb{Z})^m$.  
 It follows that the time-average over this (quasi-)periodic flow, of a complex exponential vector field $f_k e^{i\langle k, \phi\rangle}$ (with $f_k\in \mathbb{C}^m$) is given by 
\begin{align}\nonumber 
 \lim_{T\to \infty} \frac{1}{T} \int_0^T & f_ke^{i\langle k, \phi + \omega t \rangle} \, dt    = \left\{ \begin{array}{ll} f_k e^{i\langle k, \phi\rangle} & \mbox{when}\  \langle k, \omega \rangle  = 0 \, , \\ 0 & \mbox{when} \ \langle k, \omega \rangle  \neq 0 \, . \end{array} \right. 
\end{align}
This shows that $f_k e^{i\langle k, \phi\rangle}$ is resonant (that is: it depends on a resonant combination angle) precisely when it is equal to its average over the (quasi-)periodic flow, whereas $f_ke^{i\langle k, \phi\rangle}$ is nonresonant precisely when this average is zero. 
For an arbitrary (and sufficiently regular) Fourier series
it follows that 
$$\lim_{T\to\infty} \frac{1}{T} \int_0^T \left( \sum_{k\in\mathbb{Z}^m} f_k\, e^{i\langle k, \phi+\omega t\rangle} \right) dt = \sum_{\footnotesize \begin{array}{c} k\in \mathbb{Z}^m \\ \langle \omega, k\rangle = 0 \end{array}} \!\!\!\! f_k \, e^{i\langle k, \phi\rangle}\, .$$
We conclude that averaging a Fourier series removes its nonresonant terms, while keeping its resonant terms untouched.
Corollary \ref{normalformcorollary} shows that it can be arranged that the term ${\bf f}_j$ in the reduced vector field ${\bf f}$ is a sum of resonant terms only (to arbitrarily high order). We may thus loosely interpret Corollary \ref{normalformcorollary}  as a high-order averaging theorem, see \cite{sanvermur}.
\end{remk}

\begin{remk}
We include the following result for completeness. Applied to $A=e_0'(\phi)$ and $B=N(\phi)$ it gives a formula for the projection $\pi(\phi)$ onto the tangent space to $\mathbb{T}_0$ at $e_0(\phi)$  along the fast fibre at that point. This formula is not only useful for practical computations, but also shows explicitly that $\pi(\phi)$  depends smoothly on $\phi$. A proof of Proposition \ref{projectionprop} is given in \cite{BobIanMartin}.  
\end{remk}
\begin{prop}\label{projectionprop}
Let $1\leq m \leq M$ and assume that $A \in \mathcal{L}(\mathbb{R}^m, \mathbb{R}^M)$ and $B \in \mathcal{L}(\mathbb{R}^{M-m}, \mathbb{R}^M)$ are linear maps satisfying 
$\mathbb{R}^M = {\rm im}\, A \oplus {\rm im}\, B$. We denote by $\pi\in \mathcal{L}(\mathbb{R}^M, \mathbb{R}^M)$ the ``oblique projection'' onto the image of $A$ along the image of $B$, i.e., $\pi$ is the unique linear map  satisfying $\pi A = A$ and $\pi B=0$. Then $\pi$ is given by the formula
$$\pi =  A (A^T\pi(B)^{\perp} A)^{-1}A^T\pi(B)^{\perp}  \ \mbox{in which}\ \pi(B)^{\perp} := (1 - B(B^TB)^{-1}B^T)\, .$$
The $T$ denotes matrix transpose. All the inverses in this formula exist. Note that $\pi(B)^{\perp}$ is the orthogonal projection onto $\ker B^T$ along ${\rm im}\, B$.
\end{prop}

\section{Reducibility for oscillator systems}\label{Floquetsection}
In this section we show that the invariant torus of a system of uncoupled oscillators (see the introduction) is reducible. We also give a formula for the fast fibre map for such a torus. The results in this section are a consequence of Floquet's theorem, which implies that the invariant circle defined by a single hyperbolic periodic solution of an ODE is reducible. The results in this section should thus be considered well-known, but for completeness we include them in detail. We start with the result for single hyperbolic periodic orbits.

\begin{thr}\label{floquetreducible}
Let $X: \mathbb{R} \to \mathbb{R}^M$ be a hyperbolic $T$-periodic orbit of a smooth vector field ${\bf  F}: \mathbb{R}^M\to \mathbb{R}^M$. Then the invariant circle $\mathbb{T}_0=X(\mathbb{R}) \subset \mathbb{R}^M$ is reducible and normally hyperbolic. Its fast fibre map is given by formula \eqref{fastmapperiodic}.
\end{thr}
\begin{proof}
Assume that the ODE 
$\dot x = {\bf F}(x)\ \mbox{on}\ \mathbb{R}^{M}$ 
possesses a hyperbolic periodic orbit 
$X = X(t)$ with minimal period $T>0$. We think of it as an invariant circle $\mathbb{T}_0$ embedded by the map $e_0: \mathbb{R}/2\pi\mathbb{Z} \to \mathbb{R}^M$ defined by 
$e_0(\phi) := X(\omega^{-1} \phi)$, where $\omega:=\frac{2\pi}{T}$. 
Let $\Phi = \Phi(t) \in {\rm GL}(\mathbb{R}^{M})$ be the principal fundamental matrix solution of the linearisation around this periodic orbit.
This means that
$$ \dot \Phi(t) = {\bf   F}'(X(t))\cdot \Phi(t) \  \mbox{and}\ \Phi(0)= {\rm Id}_{\mathbb{R}^M}     \, .$$
Floquet's theorem  \cite{floquet, teschl} states that $\Phi(t)$ admits a factorisation
$$\Phi(t) = P(t) e^{Bt}\ \mbox{with}\ P(t+T) = P(t)\ \mbox{and}\ P(0)= {\rm Id}_{\mathbb{R}^M} \, .$$ 
The constant (and perhaps complex) Floquet matrix  $B$ satisfies $e^{BT} = \Phi(T)$, for example $B=\frac{1}{T}\log \Phi(T)$ for a choice of matrix logarithm. Note that a matrix logarithm of $\Phi(T)$ exists because $\Phi(T)$ is invertible. We shall assume here that $B$ is a real matrix. This can always be arranged by replacing $T$ by $2T$ and considering a double cover of $\mathbb{T}_0$ if necessary, but we ignore this (somewhat annoying) subtlety here. 

Substituting the Floquet decomposition in the definition of the fundamental matrix solution, we obtain that $\dot P(t) e^{Bt} + P(t) B e^{Bt} = {\bf   F}'(X(t)) P(t) e^{Bt}$. Thus,  
$$\dot P(t)+ P(t) B = { \bf  F}'(X(t)) P(t)  \, .$$
This implies that we found a solution to Equation \eqref{N0M0} in Lemma \ref{variationalparametrisation}.  Indeed, if we define  
$$\tilde L=B \ \mbox{and}\ \tilde N(\phi) = P(\omega^{-1}\phi) \, ,  $$
then we have, recalling  that $e_0(\phi)=X(\omega^{-1}\phi)$), 
\begin{align} \nonumber
\partial_{\omega}\tilde N(\phi) + \tilde N (\phi)\cdot \tilde L & = \tilde N'(\phi)\cdot  \omega + \tilde N(\phi) \cdot \tilde L  \\ \nonumber
& = \dot P(\omega^{-1} \phi) + P(\omega^{-1}\phi) \cdot B \\ \nonumber
& =  {\bf   F}'(X(\omega^{-1}\phi)) \cdot P(\omega^{-1}\phi)  = { \bf F}'(e_0(\phi))\cdot \tilde N(\phi)\, .
\end{align}
However, this does not yet prove that the periodic orbit is reducible, because $\tilde N=\tilde N(\phi)$ defines a family of $M\times M$-matrices, and hence the image of $\tilde N(\phi)$ is not normal to the tangent vector $\omega e_0'(\phi) = \dot X(\omega^{-1}\phi) $ to the periodic orbit. 
To resolve this issue, recall that $\Phi(T)$ always has a unit eigenvalue. This follows from 
 differentiating the identity $\dot X(t) = {\bf   F}(X(t))$ to $t$, which gives that $\frac{d}{dt} \dot X(t) = {\bf  F}'(X(t))\cdot  \dot X(t)$, so that 
$$\dot X(0)  = \dot X(T) =  \Phi(T)\cdot \dot X(0)\, .$$ 
Because $\Phi(T) = e^{BT}$, we conclude that $B$ has a purely imaginary eigenvalue in $\frac{2\pi i }{T} \mathbb{Z}$. Our assumption that $X$ is hyperbolic implies that none of the other eigenvalues of $B$ lie on the imaginary axis. Because $B$ is real and its eigenvalues must thus come in complex conjugate pairs, we conclude that the purely imaginary eigenvalue of $B$ must in fact be zero.  

We now choose an injective linear map $A: \mathbb{R}^{M-1}\to \mathbb{R}^{M}$ whose image coincides with the $(M-1)$-dimensional image of $B$. For any such choice of $A$ there is a unique map $L: \mathbb{R}^{M-1}\to\mathbb{R}^{M-1}$ for which
$$A \cdot L = B \cdot A\, . $$
Clearly, the eigenvalues of $L$ are the nonzero eigenvalues of $B$, showing that $L$ is hyperbolic. We also define $N: \mathbb{R}/2\pi\mathbb{Z} \to \mathcal{L}(\mathbb{R}^{M-1}, \mathbb{R}^{M})$ by
\begin{align}\label{fastmapperiodic}
N(\phi) := P(\omega^{-1}\phi) A \, .
\end{align}
By definition, ${\rm im}\, N(0)={\rm im}\, A = {\rm im}\, B$ is transverse to the tangent vector $\dot X(0) \in \ker B$ to the periodic orbit. Because each $P(t)$ is invertible, this transversality persists along the entire orbit. Indeed, writing $t = \omega^{-1}\phi$, note that 
$$ \dot X(t) = \Phi(t) \dot X(0) = P(t) e^{Bt} \dot X(0) =  P(t) \dot X(0) \in P(t)({\rm ker}\, B)$$ 
is transversal to ${\rm im}\, N(\phi) = {\rm im}(P(t)A) = P(t) ({\rm im}\, B)$. Finally, we compute 
\begin{align} \nonumber
\partial_{\omega}N(\phi) + N(\phi)L & = N'(\phi)\,  \omega + N(\phi)L   \\ \nonumber
& = \dot P(\omega^{-1} \phi) A + P(\omega^{-1}\phi) A L \\ \nonumber
& = \dot P(\omega^{-1} \phi) A + P(\omega^{-1}\phi) B A
\\ \nonumber
& =  {\bf F}'(X(\omega^{-1}\phi)) P(\omega^{-1}\phi)A  = {\bf  F}'(e_0(\phi)) N(\phi)\, .
\end{align}
This proves that the invariant circle $\mathbb{T}_0$ defined by $X(t)$ is reducible.
\end{proof}

\begin{ex}\label{exstuartlandau}
As an example consider a single Stuart-Landau oscillator
\begin{align} \label{stuartlandau}
\dot z  = \left(\alpha + i \beta \right) z +  \left(\gamma + i \delta  \right) |z|^2  z \quad \mbox{for}\ z \in \mathbb{C} \cong \mathbb{R}^2 \, .
\end{align}
Here $\alpha, \beta, \gamma, \delta \in \mathbb{R}$ are parameters. We assume that $\alpha \gamma < 0$ and $\alpha \delta - \beta\gamma \neq 0$, so that \eqref{stuartlandau} possesses a unique (up to rotation) circular periodic orbit 
$$X(t) = R e^{i\omega t} \ \mbox{where}\ R := \sqrt{-\alpha /\gamma } \ \mbox{and} \ \omega  := \beta - \alpha \delta / \gamma \neq 0\, .$$ 
Thus, the embedding 
$$e_0: \mathbb{R}/2\pi\mathbb{Z} \ni \phi  \mapsto z := R \, e^{i \phi} \in \mathbb{C} $$
sends solutions of $\dot \phi = \omega$  on $\mathbb{R}/2\pi\mathbb{Z}$ to solutions of \eqref{stuartlandau}.
The Floquet decomposition of the fundamental matrix solution around this periodic orbit can be found by anticipating that $P(t)=e^{i\omega t}$ and thus making the ansatz 
$$\Phi(t) = e^{i\omega t} e^{Bt}\, $$
for an unknown linear map $B: \mathbb{C} \to \mathbb{C}$. 
With this in mind we expand solutions to \eqref{stuartlandau} nearby the periodic orbit as 
$$z(t)=R\, e^{i\omega t} + \varepsilon \, e^{i\omega t} v(t) \, .$$
To first order in $\varepsilon$ this gives the  linear differential equations
$$\dot v = \dot v_1 + i \dot v_2 = 2 R^2 (\gamma + i \delta) v_1\, ,$$
which shows that the Floquet map $B: \mathbb{C} \to \mathbb{C}$ must be given by
$$B(v_1 + i v_2)= 2 R^2 (\gamma + i \delta) v_1 \, . $$  
This $B$ has an eigenvalue $0$ (with eigenvector $i$ corresponding to the tangent space to the invariant circle) and an eigenvalue $2\gamma R^2 = - 2 \alpha \neq 0$ (with eigenvector $\gamma +i\delta$). 
We conclude that the map
$${\bf N}e_0: (\phi, u) \mapsto (R e^{i\phi}, e^{i\phi} (\gamma + i \delta) u) \ \mbox{from}\ \mathbb{R}/2\pi\mathbb{Z} \times \mathbb{R} \ \mbox{to}\ \mathbb{C}\times \mathbb{C}$$
sends solutions of 
$$\dot \phi = \omega\, , \, \dot u = - 2\alpha \, u \ \mbox{for} \ \phi\in\mathbb{R}/2\pi\mathbb{Z} \ \mbox{and}\ u \in \mathbb{R} $$
to solutions of the linearised dynamics of \eqref{stuartlandau} on $\mathbb{C}\times \mathbb{C}$ around the invariant circle. In particular, we have $L= - 2 \alpha$ and 
$N(\phi)= e^{i\phi} ( \gamma + i \delta )$. The projection onto the tangent bundle of the invariant circle along its fast fibre bundle is given by the formulas
$$\pi(0)\cdot (x + i y) = i (y -  (\delta / \gamma) x)\ \mbox{and}\ \pi(\phi) = e^{i\phi}\cdot \pi(0)\cdot e^{-i\phi} \, .$$
Indeed, it is easy to check that $\pi(\phi)\cdot i e^{i\phi} = i e^{i\phi}$ and $\pi(\phi)\cdot e^{i\phi} ( \gamma + i \delta )  = 0$. 
\end{ex} 
\noindent 
We now extend the result of Theorem \ref{floquetreducible} to systems of multiple uncoupled oscillators, that is, systems of the form
$$\dot x_1 = F_1(x_1)\, , \, \ldots\, ,\,  \dot x_m = F_m(x_m)\, \ \mbox{with}\ x_j\in \mathbb{R}^{M_j}\, ,$$
that each have a hyperbolic $T_j$-periodic orbit $X_j(t)$. Recall that the product of these periodic orbits forms an invariant torus. The fact that this torus is reducible follows from the following lemma. Its proof is straightforward, but included here for completeness.
\begin{lem} \label{productreduciblelemma}
Let $\mathbb{T}_1\subset \mathbb{R}^{M_1}$ and $\mathbb{T}_2\subset \mathbb{R}^{M_2}$ be embedded reducible normally hyperbolic (quasi-)periodic invariant tori for the vector fields ${\bf F}_1$ and ${\bf F}_2$ respectively. Then the product torus $\mathbb{T}_0:=\mathbb{T}_1\times \mathbb{T}_2 \subset  \mathbb{R}^M$ (with $M:=M_1+M_2$) is an embedded reducible normally hyperbolic quasi-periodic invariant torus for the product vector field ${\bf F}_0$ on $\mathbb{R}^{M}$ defined by ${\bf F}_0(x_1, x_2) :=({\bf F}_1(x_1), {\bf F}_2(x_2))$.
\end{lem}

\begin{proof}
Assume that $e_j:(\mathbb{R}/2\pi\mathbb{Z})^{m_j}\to \mathbb{R}^{M_j}$ (for $j=1, 2$) is an embedding of a reducible normally hyperbolic  (quasi-)periodic invariant torus for the vector field ${\bf F}_j$. This means that there are  frequency vectors $\omega_j\in \mathbb{R}^{m_j}$ such that 
$\partial_{\omega_j}e_j = {\bf F}_j \circ e_j$ and fast fibre maps 
 ${\bf N}e_j: (\mathbb{R}/2\pi\mathbb{Z})^{m_j} \times \mathbb{R}^{M_j-m_j} \to \mathbb{R}^{M_j} \times \mathbb{R}^{M_j}$ of the form ${\bf N}e_j(\phi_j, u_j) = (e_j(\phi_j), N_j(\phi_j)\cdot u_j)$ satisfying $\partial_{\omega_j}N_j + N_j \cdot L_j = ( {\bf F}_j'\circ e_j) \cdot N_j$ for certain hyperbolic Floquet matrices $L_j$.

If we now define $m:=m_1+m_2$, $\omega:=(\omega_1, \omega_2)\in \mathbb{R}^{m}$ and $e_0: (\mathbb{R}/2\pi\mathbb{Z})^{m} \to \mathbb{R}^{M}$ by $e_0(\phi) = e_0(\phi_1, \phi_2) := (e_1(\phi_1), e_2(\phi_2))$, then $e_0$ is clearly an embedding of $\mathbb{T}_0$ and the equality $\partial_{\omega}e_0 = {\bf F}_0 \circ e_0$ holds. In other words, the product torus $\mathbb{T}_0$ is an embedded quasi-periodic invariant torus for ${\bf F}_0$.

If we also define $N(\phi) \cdot u = N(\phi_1, \phi_2) \cdot (u_1, u_2) := (N_1(\phi_1)\cdot u_1, N_2(\phi_2)\cdot u_2)$, then clearly $N(\phi)$ is injective, and therefore ${\bf N}e_0: (\mathbb{R}/2\pi\mathbb{Z})^{m} \to \mathbb{R}^M \times \mathbb{R}^M$ defined by 
$${\bf N}e_0((\phi_1, \phi_2), (u_1, u_2)) = (e_0(\phi_1, \phi_2), N(\phi_1, \phi_2)\cdot(u_1, u_2))$$ is a fast fibre map for $\mathbb{T}_0$ that satisfies  
$\partial_{\omega}N + N \cdot L = ({\bf F}_0' \circ e_0)\cdot N$. Here  $L: \mathbb{R}^{M-m}\to \mathbb{R}^{M-m}$ is defined by $L(u_1, u_2) := (L_1 u_1, L_2 u_2)$. This $L$ is hyperbolic, its eigenvalues being those of $L_1$ and $L_2$. This proves that $\mathbb{T}_0$ is reducible and normally hyperbolic and concludes the proof of the lemma.
\end{proof}

\section{Application to remote synchronisation} \label{sec:examplesection}
In this final section we apply and illustrate our phase reduction method in a small network of three weakly linearly coupled Stuart-Landau oscillators
\begin{align}\label{3stuartlandau}
\begin{array}{llll}
\dot{z}_1  = &  \hspace{-2mm} (\alpha + i\beta)z_1  &  \hspace{-2mm} +   \hspace{1mm} (\gamma + i\delta)|z_1|^2z_1  &  \hspace{-2mm} +  \hspace{1mm}  \varepsilon z_2 \, , \\ 
\dot{z}_2  = &  \hspace{-2mm} (a \, + \, \! ib)z_2  &  \hspace{-2mm} +  \hspace{1mm}  (c + id)|z_2|^2z_2   &   \hspace{-2mm} + \hspace{1mm}   \varepsilon z_1 \,  , \\
\dot{z}_3  =  &  \hspace{-2mm} (\alpha + i\beta)z_3  &  \hspace{-2mm} + \hspace{1mm}  (\gamma + i\delta)|z_3|^2z_3  &  \hspace{-2mm} +\hspace{1mm}   \varepsilon z_2  \, ,
\end{array}
\end{align}
with $z_1, z_2, z_3 \in \mathbb{C}$. Figure \ref{pictoscnetw} depicts the coupling architecture of this network. Note that the first and third oscillator in equations \eqref{3stuartlandau} are identical. We choose parameters so that each uncoupled oscillator  has a nonzero hyperbolic periodic orbit, with frequencies $\omega_1=\omega_3 \neq \omega_2$.  These periodic orbits form a $3$-dimensional invariant torus $\mathbb{T}_0$ for the  uncoupled system, which  persists as a perturbed torus $\mathbb{T}_{\varepsilon}$ for small nonzero coupling.  

Despite that fact that the first and third oscillators in \eqref{3stuartlandau} are not coupled directly, 
a numerical study of equations \eqref{3stuartlandau}   reveals that these oscillators synchronise when appropriate parameter values are chosen, see Figure \ref{firstnum}. This ``remote synchronisation'' appears to be mediated by the second oscillator, which allows the two other oscillators to communicate. Figure \ref{firstnum-c} demonstrates, again numerically, that the timescale of remote synchronisation is of the order $t\sim \varepsilon^{-2}$. This suggests that proving the synchronisation rigorously would require second-order phase reduction.

In \cite{remotestar}, remote synchronisation of Stuart-Landau oscillators was observed numerically for the first time.  A first rigorous proof of the phenomenon, for a chain of three Stuart-Landau oscillators, occurs in \cite{rosenblum}. The proof in that paper employs the high-order phase reduction method developed in \cite{Gengel}. However, the method in \cite{Gengel} does not yield the reduced phase equations in normal form. As a result, the  timescale $t\sim \varepsilon^{-2}$ is not observed in \cite{rosenblum}.

Here we apply the parametrisation method developed in this paper, to prove that the first and third oscillator in \eqref{3stuartlandau} synchronise over a  timescale $t\sim \varepsilon^{-2}$. 
 We are also able to determine how the parameters in \eqref{3stuartlandau} influence this synchronisation.  To this end, we will compute an asymptotic expansion of an embedding $e:(\mathbb{R}/2\pi\mathbb{Z})^3 \to \mathbb{C}^3$ and a 
 reduced phase vector field ${\bf f}: (\mathbb{R}/2\pi\mathbb{Z})^3 \to \mathbb{R}^3$ to second order in the small parameter. As we are primarily interested in the synchronisation of the first and third oscillator, we do not calculate the full reduced phase vector field. Instead, we only explicitly compute an evolution equation for the resonant combination angle $\Phi:=\phi_1 -  \phi_3$.  
We will show that 
\begin{align}\label{phi1minphi3}
\dot \Phi  = \varepsilon^2 \left( -A \sin \Phi  +B (1 -  \cos \Phi )  \right) + \mathcal{O}(\varepsilon^3)  \, ,
\end{align}
in which the constants $A$ and $B$ are given by the formulas
  \begin{align}\label{AB} \begin{array}{ll}
 A & \!\!\! = \frac{1}{4a^2+(\omega_1-\omega_2)^2} \left( \frac{\delta}{\gamma}(\omega_2-\omega_1) + a \left( 1 + 
\frac{d \delta}{c \gamma} \right)   + 2a^2 \left(\frac{d}{c} + \frac{\delta}{\gamma} \right)   \frac{1}{\omega_2-\omega_1} \right) \, ,
 \\ 
 B  & \!\!\! = \frac{1}{4a^2+(\omega_1-\omega_2)^2}\left( (\omega_2-\omega_1) + a \left(\frac{d}{c} - 
\frac{\delta}{\gamma} \right)   + 2a^2 \left(1 - \frac{d \delta}{c\gamma} \right)   \frac{1}{\omega_2-\omega_1}  \right) \, .
\end{array} 
 \end{align}

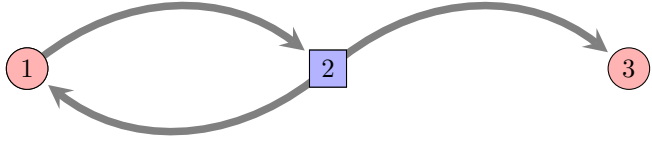
\begin{figure}[h] 
\centering{
\begin{tikzpicture}
	\node[circle,draw=black, fill=red!30, fill opacity = 1, inner sep=2.7pt, minimum size=14pt] (1) at (0,0) {$1$};
	\node[rectangle,draw=black, fill=blue!30, fill opacity = 1, inner sep=2.7pt, minimum size=14pt] (2) at (4,0) {$2$};
	\node[circle,draw=black, fill=red!30, fill opacity = 1, inner sep=2.7pt, minimum size=14pt] (3) at (8,0) {$3$};
	\draw [->,  >=stealth,  black!50, shorten <=-2pt, shorten >=2pt, 	line width=1mm] (1) to [bend right = -37.9] (2);
	\draw [->,  >=stealth,  black!50, shorten <=-2pt, shorten >=2pt, 	line width=1mm] (2) to [bend right = -37.9] (1);
	\draw [->,  >=stealth,  black!50, shorten <=-2pt, shorten >=2pt, 	line width=1mm] (2) to [bend right = -37.9] (3);
	\node[rectangle,draw=black, fill=blue!30, fill opacity = 1, inner sep=2.7pt, minimum size=14pt] (2c) at (4,0) {$2$};
	\node[circle,draw=black, fill=red!30, fill opacity = 1, inner sep=2.7pt, minimum size=14pt] (1c) at (0,0) {$1$};
\end{tikzpicture}
\caption{Representation of the network of  Stuart-Landau oscillators \eqref{3stuartlandau}.}
\label{pictoscnetw}}
\end{figure}
\noindent 
Before we prove formulas  \eqref{phi1minphi3} and \eqref{AB}, let us investigate their dynamical implications. After rescaling time $t\mapsto \tau := \varepsilon^{2} t$,   equation  \eqref{phi1minphi3} becomes
 $$\frac{d\Phi}{d\tau}  = \left( -A \sin \Phi   + B(1- \cos \Phi )  \right) + \mathcal{O}(\varepsilon)  \, .$$
For $\varepsilon=0$ the time-rescaled reduced flow on $(\mathbb{R}/2\pi\mathbb{Z})^3$  therefore admits a $2$-dimensional invariant torus 
$$S = \{\phi_1=\phi_3\} \subset (\mathbb{R}/2\pi\mathbb{Z})^3$$ 
on which the phases of the first and third oscillator   are synchronised. This torus is stable when $A>0$ and unstable when $A<0$. For $A\neq 0$, there also exists exactly one $2$-dimensional invariant torus of the form 
$$P = \{\phi_1 = \phi_3 + c\}\subset (\mathbb{R}/2\pi\mathbb{Z})^3 \ \mbox{for some}\ c\neq 0\,  $$
with the opposite stability type. The phases of the first and third oscillator are phase-locked but not synchronised on $P$.  F\'enichel's theorem guarantees that both $S$ and $P$ persist as invariant submanifolds of $(\mathbb{R}/2\pi\mathbb{Z})^3$ for small $\varepsilon\neq 0$. Hence, so do their images $e(S), e(P) \subset \mathbb{T}_{\varepsilon}\subset \mathbb{C}^3$ as invariant manifolds for \eqref{3stuartlandau}. 

For small $\varepsilon \neq 0$, a typical solution of \eqref{3stuartlandau} will therefore first converge to the $3$-dimensional invariant torus $\mathbb{T}_{\varepsilon}$ on a timescale of the order $t \sim 1$.  It will subsequently converge to either $e(S)$ or $e(P)$ on the much longer timescale $t\sim \varepsilon^{-2}$, and it is this slow dynamics that governs the synchronisation of the first and third oscillator. This multiple timescale dynamical process is illustrated in Figure \ref{firstnum}. Figure \ref{firstnum-c} confirms numerically that the timescale of synchronisation of $z_1$ and $z_3$ is indeed of the order $\varepsilon^{-2}$.

 \begin{remk}\label{remk:easyshapeA}
 We point out that the parameters in \eqref{3stuartlandau} can be tuned so that either of the two low-dimensional tori $S$ or $P$ is the stable one. Assume for instance that $\alpha, a >0$ and $\gamma, c<0$, so that $\mathbb{T}_{0}$ (and hence $\mathbb{T}_{\varepsilon}$) is stable. If in addition we choose the parameters so that $c \delta + d \gamma = 0$, then the expression for $A$ simplifies to 
 $\frac{a + (b-\beta)(\delta/\gamma) + \alpha (\delta/\gamma)^2 }{4a^2 + (\omega_1-\omega_2)^2}$. If $\delta \neq 0$, then it is clear that we can make this both positive and negative, for instance by varying the parameter $b$. Interestingly, this shows  that  properties of the second oscillator may determine whether the first and third oscillator converge to the synchronised state $S$ or the phase-locked state $P$. 

\end{remk}

\subsubsection*{Numerics}

\begin{figure}[h!]
\begin{subfigure}{0.48\linewidth}
\centering
\includegraphics[width=\linewidth]{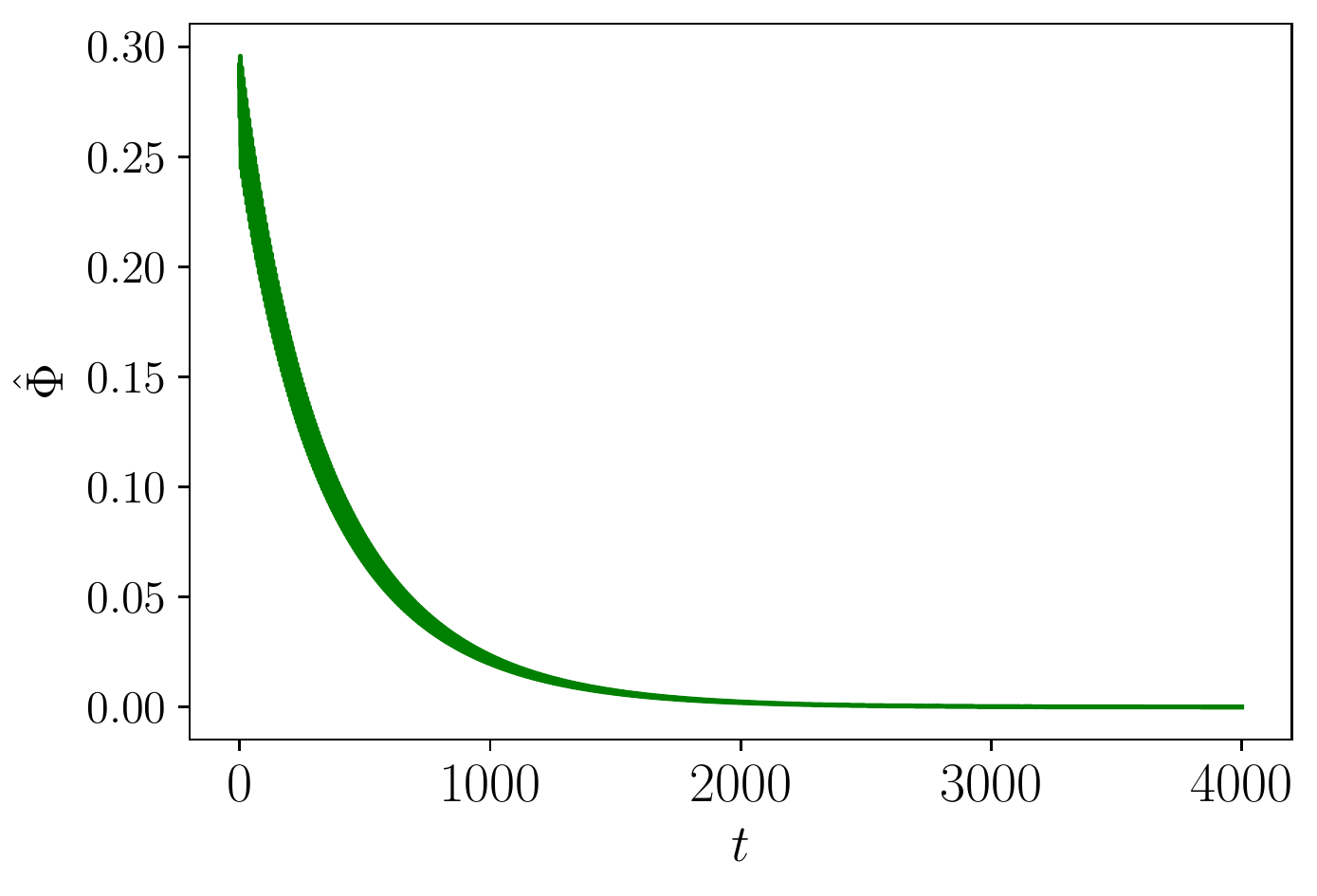}
\caption{Slow convergence of $\hat \Phi$ to zero.}
\label{firstnum-a}
\end{subfigure}\hfill 
\begin{subfigure}{0.48\linewidth}
\centering
\includegraphics[width=\linewidth]{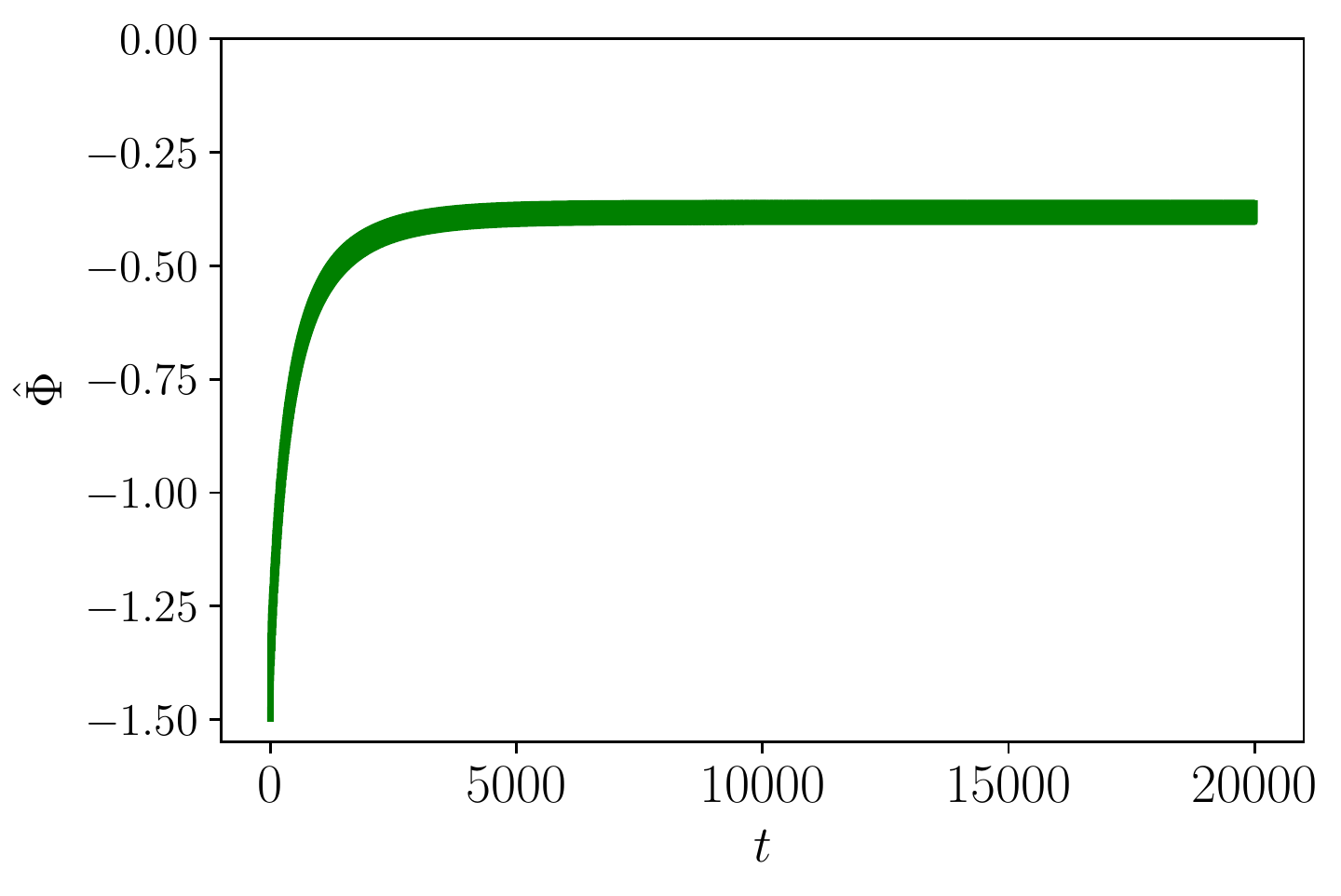}
\caption{Convergence of $\hat \Phi$ to a non-zero value.}
\label{firstnum-b}
\end{subfigure}%
\caption{Numerically obtained plots of the phase-difference $\hat \Phi = \Arg(z_1\overline{z_3}) \approx \phi_1 - \phi_3$ against time, for two different realisations of system \eqref{3stuartlandau}. }
\label{firstnum}
\end{figure}

\noindent Before proving  \eqref{phi1minphi3} and \eqref{AB}, we present some numerical results on system \eqref{3stuartlandau}.   Figure \ref{firstnum} shows numerically obtained plots of $\hat \Phi = \Arg(z_1\overline{z_3})$ against time, for two different realisations of system \eqref{3stuartlandau}.  We use $\hat \Phi$ as a proxy for $\Phi = \phi_1 - \phi_3$. As this approximation does not take into account the distortion of the perturbed invariant torus, we observe small amplitude, rapid oscillations in $\hat \Phi$, causing the lines in Figure \ref{firstnum} to be  thick.
In Figure \ref{firstnum-a}, we have chosen the parameter values  
\begin{equation}\label{convergentvalues}
\begin{array}{llll}%
\alpha = 1  & \beta= 1  &  \gamma = -1 &  \delta = 1\, ;\\
a = 1  & b= 2  &  c = -1 &  d = -1\, ,
\end{array}
\end{equation}
together with $\varepsilon = 0.1$. 
It follows that $c \delta + d \gamma  = 0$, and so $A = \frac{1}{5} >0$, 
see Remark \ref{remk:easyshapeA}.
The  above analysis  therefore predicts that $\hat \Phi$ should converge to zero, which the figure indeed shows.
The convergence is very slow, as only around $t= 2000$ do we find that $\hat \Phi$ is indistinguishably close to zero. 
We will comment more on the rate of convergence below.
Figure \ref{firstnum-a} was generated using Euler's method with time steps of $0.05$, starting from the point in phase space $(z_1, z_2, z_3) = (-1,1+0.4i, -1+0.3i) \in \C^3$.

For Figure \ref{firstnum-b} we have likewise set $\varepsilon = 0.1$, but have instead chosen 
\begin{equation}
\begin{array}{llll}%
\alpha = 1  & \beta= 0.1  &  \gamma = -1 &  \delta = 1\, ;\\
a = 1  & b= 6  &  c = -1 &  d = -1\, ,
\end{array}
\end{equation}
which yields 
$
A = 
\frac{-3.9 }{4 + (3.9)^2} = -0.203\ldots  < 0$. 
Hence, our theory predicts $\hat \Phi$ to converge to a non-zero constant value, which is indeed seen to be the case. 
Again the thickness of the line is due to rapid oscillations.
Figure \ref{firstnum-b} is generated in the same way as Figure \ref{firstnum-a}, except that the starting point for Euler's method is now $(z_1, z_2, z_3)  = (1+0.3i,1+0.4i, -0.2+0.9i)$.

\begin{figure}[h!]
\centering
\includegraphics[width=.7\linewidth]{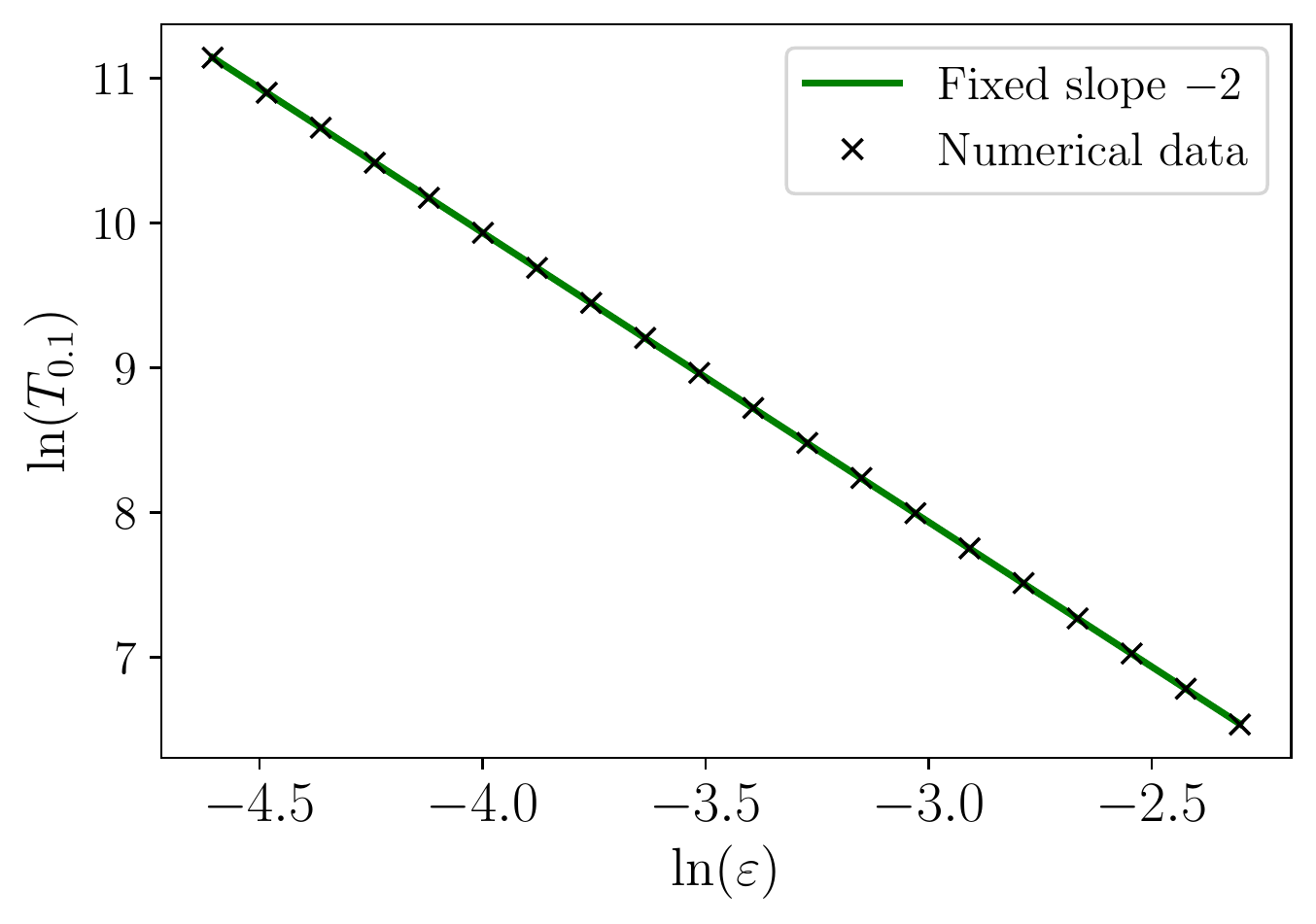}
\caption{Log-log plot of the time $T_{0.1}$ it takes for $\hat \Phi$ to decrease by a factor of $10$, against the coupling parameter $\varepsilon$.}
\label{firstnum-c}
\end{figure}

Finally, Figure \ref{firstnum-c} displays the rate of convergence to synchrony  as a function of $\varepsilon$. 
The figure was made using Euler's method with time-steps of $0.05$, all starting from the same point $(z_1, z_2, z_3)  = (-1+0.3i,1+0.4i, -1+0.5i)$.
 We have again chosen the parameters as in \eqref{convergentvalues}, so that we may expect $\hat \Phi$ to converge to zero.
However, the rate at which this occurs depends on $\varepsilon$. We measure this rate by recording $T_{0.1}$, which is the smallest time $t$ for which $|\hat \Phi(t)| \leq 0.1 |\hat \Phi(0)|$. 

Figure \ref{firstnum-c} shows a log-log plot of $T_{0.1}$ against $\varepsilon$. 
The crosses in the figure represent numerical results for $20$ different values of $\varepsilon$. 
Shown in green is the  line with slope $-2$  through the leftmost cross.
We see that 
$\ln(T_{0.1}) = -2\ln(\varepsilon) + C$ for some $C \in \R$ to very good approximation. 
Hence we find $T_{0.1} \sim \varepsilon^{-2}$, which is fully in agreement with our predictions.

\subsubsection*{Setup: the unperturbed problem}
We now start our proof of formulas \eqref{phi1minphi3} and \eqref{AB}. We first recall some observations from Example \ref{exstuartlandau}, and make assumptions on the parameters that appear in \eqref{3stuartlandau}. 
Specifically, we assume that these parameters  are chosen so that 
$$\alpha \gamma < 0, \ a c <0,\ \beta\gamma - \alpha \delta \neq 0, \ bc - a d \neq 0\,  \ \mbox{and}\ \omega_1 = \omega_3 \neq \omega_2\, .$$ 
Recall from Example \ref{exstuartlandau} that this ensures that all three uncoupled oscillators possess a unique hyperbolic periodic orbit, with nonzero frequencies $\omega_1 = \omega_3 = \beta - \alpha \delta / \gamma$ and $\omega_2 = b - a d / c \neq \omega_1$. The product of these periodic orbits forms a $3$-dimensional reducible normally hyperbolic (quasi-)periodic invariant torus $\mathbb{T}_0 \subset \mathbb{C}^3$. 
An embedding of $\mathbb{T}_0$ is given by 
$$e_0: (\mathbb{R}/2\pi\mathbb{Z})^3 \to \mathbb{C}^3 \ \mbox{defined by}\  e_0(\phi_1, \phi_2, \phi_3)  = (R_1 \, e^{i \phi_1}, R_2 \, e^{i \phi_2}, R_3 \, e^{i \phi_3})\,  $$ where $$R_1 = R_3 = \sqrt{-\alpha /\gamma }>0\ \mbox{and}\ R_2 = \sqrt{-a /c }>0\ . $$
This embedding sends integral curves of the constant vector field $\omega=(\omega_1, \omega_2, \omega_3)$ on $(\mathbb{R}/2\pi\mathbb{Z})^3$ 
to solutions of \eqref{3stuartlandau} (with $\varepsilon=0$) on $\mathbb{C}^3$. 

It follows from Example \ref{exstuartlandau} and Lemma \ref{productreduciblelemma} that a Floquet matrix for $\mathbb{T}_0$ is
$$L  = {\rm diag}(-2 \alpha, -2 a , -2\alpha)\, ,$$
 with corresponding fast fibre map given by the family of injective linear maps  $N: (\mathbb{R}/2\pi\mathbb{Z})^3 \to \mathcal{L}(\mathbb{R}^3, \mathbb{C}^3)$  defined by
 $$N(\phi_1, \phi_2, \phi_3)   = {\rm diag}(e^{i\phi_1} (\gamma + i \delta), e^{i\phi_2} (c + i d), e^{i\phi_3} (\gamma + i \delta))\, .$$   
The projection onto the tangent bundle along the fast fibre bundle is given by
$$\pi(\phi_1, \phi_2, \phi_3)  = {\rm diag}(e^{i\phi_1}\pi_1(0)e^{-i\phi_1}, e^{i\phi_2}\pi_2(0)e^{-i\phi_2}, e^{i\phi_3}\pi_3(0)e^{-i\phi_3})\, .$$
Here,
 $$\pi_1(0) (x_1+ i y_1 ) = i (y_1 - (\delta/\gamma)x_1)\, , \ \pi_2(0) (x_2+ i y_2 ) =  i (y_2 - (d/c)x_2)\, ,$$ 
 and $\pi_3(0)=\pi_1(0)$. 

\subsubsection*{The first tangential homological equation}
We now compute ${\bf f}_1$ and ${\bf g}_1$ from the first tangential homological equation, see \eqref{tworeducedequations1}, with $U_1$ as given in \eqref{tworeducedequations}. A short calculation  shows that the projection of the inhomogeneous term ${\bf G}_1(\phi) = {\bf F}_1(e_0(\phi)) = (R_2e^{i\phi_2}, R_1e^{i\phi_1}, R_2e^{i\phi_2})$ is 
 $$(\pi \cdot {\bf G}_1)(\phi) =  \left( \begin{array}{c} i R_2e^{i\phi_1} \left( \sin (\phi_2-\phi_1) - (\delta/\gamma) \cos (\phi_2-\phi_1) \right) \\ i R_1e^{i\phi_2} \left( \sin (\phi_1-\phi_2) - (d/c) \cos (\phi_1-\phi_2) \right) \\ i R_2e^{i\phi_3} \left( \sin (\phi_2-\phi_3) - (\delta/\gamma) \cos (\phi_2-\phi_3) \right)\end{array} \right)\, . $$
This is clearly in the range of $e_0'(\phi) = {\rm diag}(iR_1e^{i\phi_1}, iR_2e^{i\phi_2}, iR_3e^{i\phi_3})$.  Thus the first tangential homological equation becomes
$$\partial_{\omega} {\bf g}_1(\phi) + {\bf f}_1(\phi)  =  U_1(\phi) = \left( \begin{array}{c}     
(R_2/R_1) \left( \sin (\phi_2-\phi_1) - (\delta/\gamma) \cos (\phi_2-\phi_1) \right) 
\\
(R_1/R_2) \left( \sin (\phi_1-\phi_2) - (d/c) \cos (\phi_1-\phi_2) \right) 
\\
(R_2/R_3) \left( \sin (\phi_2-\phi_3) - (\delta/\gamma) \cos (\phi_2-\phi_3) \right) 
\end{array} \right) \, .$$
 Because $\omega_1\neq \omega_2$ we are able to choose the solutions ${\bf f}_1(\phi)  = (0, 0, 0)$ and  
$${\bf g}_1(\phi) =  \frac{1}{\omega_1-\omega_2} \left( \begin{array}{r}   
(R_2 / R_1) \left( \cos (\phi_2-\phi_1)   + (\delta/\gamma)  \sin (\phi_2-\phi_1) \right) \\
-(R_1 / R_2) \left( \cos (\phi_1-\phi_2)   + (d/c)  \sin (\phi_1-\phi_2) \right) \\
(R_2 / R_3) \left( \cos (\phi_2-\phi_3)   + (\delta/\gamma)  \sin (\phi_2-\phi_3) \right) 
\end{array} \right)\, . $$
 
 \subsubsection*{The first normal homological equation}
Another short computation allows us to express the projection $(1-\pi)\cdot {\bf G}_1$  as
$$((1-\pi)\cdot {\bf G}_1)(\phi) = \left( \begin{array}{l} 
e^{i\phi_1} ( \gamma+ i \delta) (R_2/\gamma) \cos (\phi_2-\phi_1) \\ 
e^{i\phi_2} (c+ i d) (R_1/c) \cos (\phi_1-\phi_2) \\ 
e^{i\phi_3} ( \gamma+ i \delta) (R_2/\gamma) \cos (\phi_2-\phi_3)
 \end{array} \right) \, .
$$
This is clearly in the range of $N(\phi)={\rm diag}(e^{i\phi_1}(\gamma+i\delta), e^{i\phi_2}(c+id), e^{i\phi_3}(\gamma+i\delta))$. Thus the first normal homological equation, see \eqref{tworeducedequations2} and \eqref{tworeducedequations}, reads
$$\partial_{\omega}{\bf h}_1(\phi) + {\rm diag}(2\alpha, 2a, 2\alpha){\bf h}_1(\phi) = V_1(\phi) =  \left( \begin{array}{l} 
 (R_2/\gamma) \cos (\phi_2-\phi_1) \\ 
 (R_1/c) \cos (\phi_1-\phi_2) \\ 
 (R_2/\gamma) \cos (\phi_2-\phi_3)
 \end{array} \right) \, .
$$
The solution reads 
$${\bf h}_1(\phi) = \left( \begin{array}{l} 
\frac{R_2}{\gamma(4\alpha^2 + (\omega_1-\omega_2)^2)} \left(2\alpha \cos(\phi_2-\phi_1) + (\omega_2-\omega_1)\sin(\phi_2-\phi_1) \right) 
\\ 
\frac{R_1}{c(4a^2 + (\omega_1-\omega_2)^2)} \left(2a \cos(\phi_1-\phi_2) + (\omega_1-\omega_2)\sin(\phi_1-\phi_2) \right) 
\\
\frac{R_2}{\gamma(4\alpha^2 + (\omega_1-\omega_2)^2)} \left(2\alpha \cos(\phi_2-\phi_3) + (\omega_2-\omega_1)\sin(\phi_2-\phi_3) \right) 
\end{array} \right)\, .
$$

\subsubsection*{Second order terms}
Let us clarify that we will not solve the second order homological equations completely. Instead, the only second order terms that we  compute explicitly are the first and third components ${\bf f}_2^{(1)}$ and ${\bf f}_2^{(3)}$ of the second order part ${\bf f}_2$ of the reduced phase vector field. As was explained above, this suffices to obtain the desired asymptotic expression for 
 $\frac{d}{dt}(\phi_1 - \phi_3) = \varepsilon^2 \left( {\bf f}_2^{(1)}(\phi) - {\bf f}_2^{(3)}(\phi) \right) + \varepsilon^3 \ldots $. 
  
  We first compute the inhomogeneous term ${\bf G}_2$ as given in \eqref{iterativeeqns}. Because ${\bf F}_2=0$ and ${\bf f}_1=0$, we  see that 
  $${\bf G}_2=\frac{1}{2}({\bf F}_0'' \circ e_0)(e_1, e_1) +  ({\bf F}_1' \circ e_0)\cdot e_1$$
  consists only of two terms. It also turns out that the first of these terms  contributes in a rather trivial manner to the phase dynamics at order $\varepsilon^2$. 
 This term can be computed  by making use of the expansion
   \begin{align}\nonumber 
    |R_je^{i\phi_j}+\varepsilon &  e_1^{(j)}(\phi)|^2(R_je^{i\phi_j}+\varepsilon e_1^{(j)}(\phi)) =  R^3_je^{i\phi_j} + \varepsilon R_j^2 \left(2 e_1^{(j)}(\phi) + e^{2i\phi_j}\overline{e_1^{(j)}(\phi)} \right)  \\ \nonumber  & + \varepsilon^2  R_j \left(2 e^{i\phi_j} |e_1^{(j)}(\phi)|^2 +  e^{-i\phi_j}(e_1^{(j)}(\phi))^2 \right) + \mathcal{O}(\varepsilon^3) \, .
  \end{align}
  This leads to the formula 
  \begin{align} \label{D2Omega}
 \frac{1}{2}&({\bf F}_0''(e_0(\phi)) (e_1(\phi), e_1(\phi)) \! = \nonumber \\ &
  \! \underbrace{
  \left(\!\! \begin{array}{l}   
  R_1 (\gamma+i\delta) (2 e^{i\phi_1} |e_1^{(1)}(\phi)|^2 
   \\ 
   R_2 (c+id) (2 e^{i\phi_2} |e_1^{(2)}(\phi)|^2 
   \\
   R_3 (\gamma+i\delta) (2 e^{i\phi_3} |e_1^{(3)}(\phi)|^2 
  \end{array} \!\! \right)  }_{=: T_1(\phi) \in \, {\rm im}\, N(\phi)}
  +  
 \underbrace{ \left( \!\! \begin{array}{l}  
  R_1 (\gamma+i\delta)  e^{-i\phi_1} (e_1^{(1)}(\phi))^2 
   \\
  R_2 (c+id)  e^{-i\phi_2} (e_1^{(2)}(\phi))^2 
    \\
  R_3 (\gamma+i\delta)  e^{-i\phi_3} (e_1^{(3)}(\phi))^2 
  \end{array} \!\! \right)}_{=: T_2(\phi)}
   \, .
  \end{align}
It is clear that the first term on the right hand side of \eqref{D2Omega}---which we called $T_1(\phi)$---lies in the range of $N(\phi)$ because $2 R_j |e_1^{(j)}(\phi)|^2 \in \mathbb{R}$ for $j=1,2,3$. So this first term vanishes when we apply the projection $\pi(\phi)$. 

The projection of the second term on the right hand side of \eqref{D2Omega}---which we called $T_2(\phi)$---can be computed as follows. Recall from \eqref{Ansatz} that 
$e_1(\phi) = e_0'(\phi)\cdot {\bf g}_{1}(\phi) + N(\phi)\cdot {\bf h}_1(\phi)$, where $e_0$, ${\bf g}_1$, $N$ and ${\bf h}_1$ are given in the formulas above. This can be used to expand, first the $(e_1^{(j)}(\phi))^2$, and then $T_2(\phi)$ in trigonometric polynomials.  It is not very hard to see that this must yield a formula of the form 
\begin{align}\nonumber 
& \pi(\phi)T_2(\phi)  =   \left( \begin{array}{l} R_1 i e^{i\phi_1} \left( C+ D \sin(2\phi_2-2\phi_1) + E \cos(2\phi_2-2\phi_1) \right) \\ R_2 i e^{i\phi_2} \left( \tilde C+ \tilde D \sin(2\phi_2-2\phi_1) + \tilde E \cos(2\phi_2-2\phi_1) \right) \\  R_3 i e^{i\phi_3} \left( C + D \sin(2\phi_2-2\phi_3) + E \cos(2\phi_2-2\phi_3) \right)  \end{array} \right) 
\end{align}
 for certain real numbers $C, D, E, \tilde C, \tilde D, \tilde E$ that we shall not explicitly compute here. Note that this  clearly lies in the range of $e_0'(\phi)$. 
It follows that 
\begin{align}\nonumber 
& U_2^{1{\rm st}}(\phi)  =   \left( \begin{array}{l}  C  + D \sin(2\phi_2-2\phi_1) + E \cos(2\phi_2-2\phi_1)  \\   \tilde C+ \tilde D \sin(2\phi_2-2\phi_1) + \tilde E \cos(2\phi_2-2\phi_1)  \\   C + D \sin(2\phi_2-2\phi_3) + E \cos(2\phi_2-2\phi_3)   \end{array} \right) 
\end{align}
is the first part of the inhomogeneous right hand side of the second tangential homogeneous equation $\partial_{\omega}{\bf g_2}+{\bf f}_2 = U_2$. Because $2\omega_1\neq 2\omega_2$, only the constant part $(C, \tilde C, C)$ of this $U_2^{1{\rm st}}(\phi)$ is resonant; all  other terms can be absorbed in ${\bf g}_2$. Thus the resonant normal form of this part of ${\bf f}_2$ is $(C, \tilde C, C)^T$. As this constant vector field does not contribute to $\frac{d}{dt}\left( \phi_1 -   \phi_3 \right)$, we compute neither $C$ nor $\tilde C$ explicitly. 
  
 We proceed by considering the  other term in ${\bf G}_2$, namely $({\bf F}'_1 \circ e_0)\cdot e_1$. Recalling that ${\bf F}_1(z)=(z_2, z_1, z_2)$, we see that this term equals
 $${\bf F}_1'(e_0(\phi))\cdot e_1(\phi) = \left( \begin{array}{c} e_1^{(2)}(\phi) \\ e_1^{(1)}(\phi) \\ e_1^{(2)}(\phi) \end{array} \right) = 
 \left( \begin{array}{l} e^{i\phi_2} (iR_2{\bf g}_1^{(2)}(\phi) + (c + i d){\bf h}_1^{(2)}(\phi) ) \\
 e^{i\phi_1} (iR_1{\bf g}_1^{(1)}(\phi) + (\gamma + i \delta){\bf h}_1^{(1)}(\phi) ) \\
e^{i\phi_2} (iR_2{\bf g}_1^{(2)}(\phi) + (c + i d){\bf h}_1^{(2)}(\phi) )  
  \end{array} \right) \, .$$
Using the expressions for $\pi(\phi)$, ${\bf g}_1(\phi)$ and ${\bf h}_1(\phi)$ provided above, one can compute that the projection of this term has the form 
  \begin{align}
  & \pi(\phi)\cdot  {\bf F}_1'(e_0(\phi))\cdot e_1(\phi) =   \left( \begin{array}{ccc} i R_1 e^{i\phi_1} & 0 & 0  \\  0 & i R_2 e^{i\phi_2} & 0 \\ 0 & 0& i R_3 e^{i\phi_3} \end{array} \right) \cdot U_2^{2{\rm nd}}(\phi)\, ,
  \end{align}
  in which now
  \begin{align}  
  U_2^{2{\rm nd}}(\phi) = \left( \begin{array}{l}
    B +  F\sin( 2 \phi_1 -2\phi_2) + G \cos ( 2 \phi_1 -2\phi_2)  \\
   \tilde  B +  \tilde F\sin( 2 \phi_1 -2\phi_2) + \tilde G \cos ( 2 \phi_1 -2\phi_2)  \\
 \left\{  \begin{array}{l}
A \sin(\phi_1-\phi_3) + B\cos(\phi_1-\phi_3) \\ +  F\sin( \phi_1+\phi_3-2\phi_2) + G \cos ( \phi_1+\phi_3-2\phi_2) \end{array}
  \right\} \end{array}
 \right)    \, .
   \end{align}
  With some effort the constants $A$ and $B$ can be computed by hand, yielding
    \begin{align}\label{ABagain} \begin{array}{ll}
 A &\!\!\! = \frac{1}{4a^2+(\omega_1-\omega_2)^2} \left( \frac{\delta}{\gamma}(\omega_2-\omega_1) + a \left( 1 + 
\frac{d \delta}{c \gamma} \right)   + 2a^2 \left(\frac{d}{c} + \frac{\delta}{\gamma} \right)   \frac{1}{\omega_2-\omega_1} \right) \, ,
 \\ 
 B  & \!\!\! = \frac{1}{4a^2+(\omega_1-\omega_2)^2}\left( (\omega_2-\omega_1) + a \left(\frac{d}{c} - 
\frac{\delta}{\gamma} \right)   + 2a^2 \left(1 - \frac{d \delta}{c\gamma} \right)   \frac{1}{\omega_2-\omega_1}  \right) \, .
\end{array} 
 \end{align} 
We did not compute any of the other constants. 
 As $\omega_1=\omega_3 \neq \omega_2$, the resonant part of   $U_2^{2{\rm nd}}(\phi)$ is given by 
 ${\bf f}_2(\phi) = (B, \tilde B, A\sin(\phi_1-\phi_3) + B\cos(\phi_1-\phi_3))^T$. The other terms in  $U_2^{2{\rm nd}}(\phi)$ can be absorbed into ${\bf g}_2$ when solving the tangential homological equation $\partial_{\omega}{\bf g_2}+{\bf f}_2 = U_2$.   

\subsubsection*{Conclusion}
 To summarise, we computed that ${\bf f}_1(\phi) = (0, 0,0)^T$  and 
 \begin{align}\label{f23}
{\bf  f}_2(\phi) =   \left( \begin{array}{c}  B+C \\  \tilde B +\tilde C \\ A \sin (\phi_1-\phi_3) + B \cos  (\phi_1-\phi_3)  + C  \end{array} \right) \, .
 \end{align}
 The constants $A$ and $B$ are given in \eqref{ABagain}, but we did not compute $\tilde B, C$ or $\tilde C$. 
   Because $\omega_1=\omega_3$ and $\dot \phi = \omega + \varepsilon {\bf f}_1(\phi) + \varepsilon^2 {\bf f}_2(\phi) + \mathcal{O}(\varepsilon^3)$, we conclude that 
\begin{align}   
\frac{d}{dt} (\phi_1 - \phi_3) = \varepsilon^2 \left( -A \sin (\phi_1-\phi_3) - B \cos  (\phi_1-\phi_3)  + B \right) + \mathcal{O}(\varepsilon^3)  \, .
\end{align}
  This is exactly equation \eqref{phi1minphi3}.

\section{Acknowledgements}
We would like to thank Edmilson Roque and Deniz Eroglu for useful tips regarding numerics.
S.v.d.G. was partially funded by the Deutsche Forschungsgemeinschaft (DFG, German Research Foundation)--–453112019. E.N. was partially supported by the Serrapilheira Institute (Grant No. Serra-1709-16124). 
B.R. acknowledges funding and hospitality of the Sydney Mathematical Research Institute.
  \bibliography{CoupledNetworks}
\bibliographystyle{amsplain}

  \end{document}